\documentclass{siamltex}

\pdfoutput=1

\usepackage{amsmath,amscd}
\usepackage{amssymb}
\usepackage{algorithm}
\usepackage{algpseudocode}
\usepackage{bm}
\usepackage{mathptmx}
\usepackage{graphicx}
\usepackage{floatflt}

\usepackage[square]{natbib} 
\bibpunct{[}{]}{;}{a}{}{,} 

\def\argmin{\mathop{\rm argmin}}

\newcommand{\bfi}{\bfseries\itshape}

\usepackage{color}

\newcommand{\vc}[1]{\mbox{\textbf{{$\mathsf #1$}}}}
\newcommand{\defeq}{\mbox{ \ $\overset{\text{def}}{=}$ \ }}

\newcommand{\figspace}{\vspace{-1cm}}

\newcommand{\vr}[1]{\mbox{$\bm{#1}$}}

\newtheorem{define}{Definition}

\title{Geometric Numerical Integration \\ of Inequality Constrained, Nonsmooth Hamiltonian Systems}

\author{Danny M. Kaufman$^{\dagger,\ddag}$ and Dinesh K. Pai$^\ddag$ \\ {\tiny $\dagger$ Columbia University \quad $\ddag$ University of British Columbia}}

\begin{document}

\setlength{\tabcolsep}{0.3mm}

\maketitle

\begin{abstract} We consider the geometric numerical integration of Hamiltonian systems subject to both equality and \emph{``hard'' inequality} constraints. As in the standard geometric integration setting, we target long-term structure preservation. We additionally, however, also consider invariant preservation over persistent, simultaneous and/or frequent boundary interactions. Appropriately formulating geometric methods to include such conditions has long-remained challenging due the inherent nonsmoothness they impose. To resolve these issues we thus focus both on symplectic-momentum preserving behavior \emph{and} the preservation of additional structures, unique to the inequality constrained setting. Leveraging discrete variational techniques, we construct a family of geometric numerical integration methods that not only obtain the usual desirable properties of momentum preservation, approximate energy conservation and equality constraint preservation, but also enforce multiple simultaneous inequality constraints, obtain smooth unilateral motion along constraint boundaries and allow for both nonsmooth and smooth boundary approach and exit trajectories. Numerical experiments are presented to illustrate the behavior of these methods on difficult test examples where both smooth and nonsmooth active constraint modes persist with high frequency.
\end{abstract}

\section{Introduction}

Recent attention has focused on geometric numerical integration methods that closely preserve invariants of continuous systems~\citep{Sanz94,Hairer03,Leimkuhler04}. In particular, symplectic-momentum preserving integration schemes have been successfully applied over a wide range of applications. In the unconstrained setting, these methods exactly preserve the symplectic form and momenta, and maintain good long term energetic behavior by approximately conserving energy (up to an additive constant) over long spans of simulation. 

While extensions of symplectic methods to equality constrained systems have been extensively investigated~\citep{Hairer03,Leimkuhler04}, appropriately formulating symplectic methods to include inequality constraints continues to remain a challenging problem that has 
received limited attention in the literature. 
The difficulty in treating these latter cases stems from the inherent nonsmoothness imposed by such conditions.  \citet{Stewart00}, in particular, notes that, despite good conservative properties in the smooth setting, symplectic methods do not necessarily translate easily to nonsmooth problems. Direct extensions of existing symplectic methods (e.g., implicit midpoint method, Newmark methods, 
etc.) loose their good energetic properties and often generate surprisingly nondeterministic errors when subject to inequality constraints~\citep{Stewart00}.

In this work we consider Hamiltonian systems subject to both equality and \emph{inequality} constraints. As in the standard geometric integration setting, we target long-term structure preservation. In the inequality constrained setting, however, we additionally consider 
structure preservation over persistent, simultaneous and/or frequent boundary interactions.

With these goals in mind we focus both on symplectic-momentum preserving behavior \emph{and} the preservation of additional structures, unique to the inequality constrained setting.
In particular, we develop geometric numerical integration methods that \emph{not only} obtain the usual desirable properties of momentum preservation, approximate energy conservation and equality constraint preservation, \emph{but also} enforce many simultaneous inequality constraints, obtain smooth unilateral motion along constraint boundaries and allow for both nonsmooth and smooth boundary approach and exit trajectories.

\subsection{Problem Statement}
\label{sec:problem_statement}

We consider the general case of a Hamiltonian system subject to generalized constraints that restrict system configurations, $\vc q$, to an \emph{admissible} set, $\vc A$ (i.e., $\vc q \in \vc A$). In general, this admissible set 
can be both nonconvex and nonsmooth.

More concretely, when the admissible set is given by a manifold, it will often be defined by a set of constraint equations of the form
\begin{equation}
\label{eq:equality}
\vc f(\vc q) = (f_0(\vc q), ... , f_n(\vc q))^T = 0.
\end{equation}
For these cases the resulting holonomic Hamiltonian system reduces to a set of differential algebraic equations (DAE) for which a wide range of existing geometric approaches may be suitable~\citep{Hairer03,Leimkuhler04}. 
 
In the following, we consider generalizations where the admissible set, $\vc A$, can always be described by both a set of constraint equalities (as above) \emph{and} constraint inequalities, given by functions of the configuration variables,
\begin{equation}
\label{eq:inequality}
\vc g(\vc q) = (g_0(\vc q), ... , g_m(\vc q))^T \geq 0.
\end{equation}

We start by considering the underlying, unconstrained system, with its \emph{natural} Hamiltonian\footnote{Throughout, to simplify discussion and visualization, we consider the case of separable Hamiltonians with a flat metric on configuration. } and Lagrangian functions, 
\begin{align}
H(\vc q, \vc p) = \frac{1}{2} \vc p^T \vc M^{-1} \vc p + V(\vc q) \quad \text{and} \quad 
L(\vc q, \dot{\vc q}) = \frac{1}{2} \dot{\vc q}^T \vc M \dot{\vc q} - V(\vc q).
\end{align}

We can then compactly include constraint forces by adding an additional potential term that, when extremized, will penalize all nonadmissible configurations to ensure constraint compliance. In particular, to enforce so-called ``hard'' or ``exact'' constraints it is standard to consider the most extreme penalization available; we apply the extended value indicator function that \emph{infinitely} penalizes noncompliance.
For an arbitrary admissible set $\vc A$, the indicator penalty function is given by
\begin{equation}
    I_{\vc{A}}(\vc x) = \left\{
\begin{array}{l}
        0, \> \> \vc x \in \vc{A} \\
      \infty , \> \vc x \notin \vc{A}. \end{array}  \right.
\end{equation}
The full, nonsmooth, inequality constrained, Hamiltonian and Lagrangian formulations for the system are then obtained 
by augmenting the unconstrained system's natural potential function, $V$, with the indicator function on the admissible set, $I_{\vc A}$~\citep{Clarke}.

The corresponding generalized equations of motion for the constrained system are then given by the Euler-Lagrange differential inclusion (DI), 
\begin{equation}
\label{eq:DIEL}
\vc M \ddot{\vc q} + \nabla V(\vc q) \in -\partial I _{\vc A}(\vc q).
\end{equation}
This inclusion, obtained from the application of $\partial$, the generalized gradient operator~\footnote{Note that here and in the following we reserve the $\partial$ notation to exclusively designate the generalized gradient operator~\citep{RockWets98}.}, to the nonsmooth penalty function~\citep{RockWets98}, indicates that feasible constraint forces must lie in the inward pointing normal cone to the admissible set at $\vc q$, given by $-\partial I _{\vc A}(\vc q)$. 
On the interior of the feasible set the normal cone is trivially the origin. In this case the
Euler-Lagrange DI simply reduces to the standard unconstrained equations of motion. Elsewhere, on the boundary, where constraint forces may be necessary, the normal cone is 
is given by a nontrivial, negative span of the active constraint gradients at $\vc q$. See Figure \ref{fig:admissibleSet}. 

\begin{figure}
[h]
\centering 
\hspace{-14mm}
\includegraphics[height=0.25\hsize]{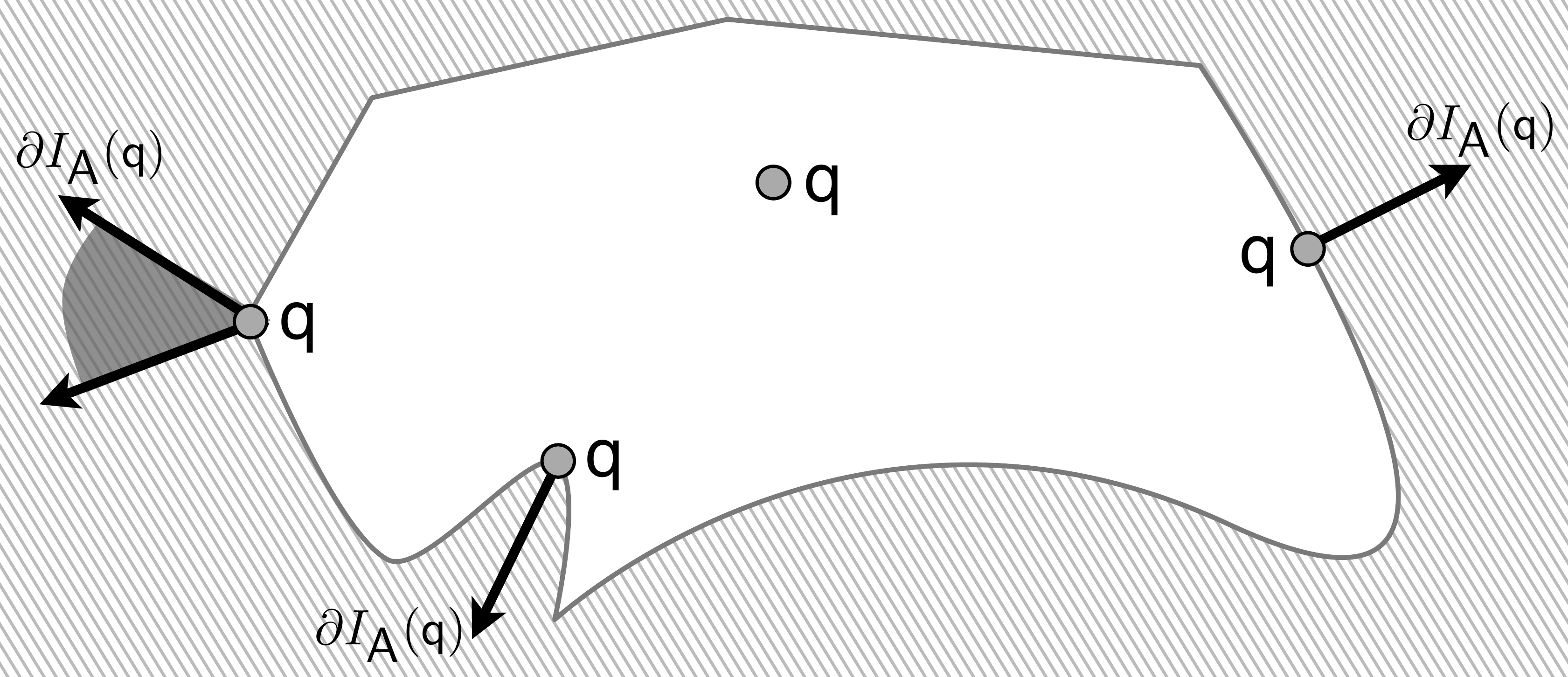}
\caption{ An example of an admissible set and a few evaluations of its associated normal cones. \label{fig:admissibleSet}}
\end{figure}

Our problem is then to derive geometric numerical schemes that integrate the generalized Euler-Lagrange system in  (\ref{eq:DIEL}) forward in time while preserving momenta \emph{and} maintaining the good approximate energy preserving properties that are standard for symplectic integration schemes in the unconstrained setting~\citep{Hairer03}. 
In the inequality constrained setting, however, there is additional critical behavior to consider:

\begin{enumerate}
\item \textbf{ Smooth trajectories along the admissible set boundary.}  
Consider, as a simple example, a mass-particle moving with a tangential velocity across a smooth region of the boundary (Figure \ref{fig:traj_ex} (a)). If a force is pushing
the particle into the boundary, the particle's trajectory should remain smooth along the boundary. No upward bounce or other type of normal motion should result. Likewise, unless dissipative forces are additionally applied, constraints should not apply ``boundary capturing''-type forces that slow the particle's tangential motion. Nevertheless, existing symplectic methods generally fail in one or both of these categories. 
This has important implications 
since, in the general setting,
many important physical phenomena are covered by the abstraction of smooth boundary motion. Such smooth boundary systems include oscillatory elastica resting on foundations, twisted elastic knots (such as DNA filaments), mechanical and biological grasping models, masonry structures, granular flow, as well as many other boundary relative systems that are important across structural engineering, robotics, mechanics, biology and other critical application areas.

\item \textbf{Nonsmooth exit and approach trajectories.}
Many trajectories in the hard constraint context \emph{can not be resolved by finite forces}. For instance, consider a mass-particle approaching a smooth boundary in the normal direction (Figure \ref{fig:traj_ex}  (b)). To resolve the impact, the mass-particle's trajectory must be reflected about the surface normal and thus we require a nonsmooth transition. More generally, whenever a trajectory is non-tangential to the admissible set boundary, an instantaneous nonsmooth jump in state is demanded. Nevertheless many symplectic-based approaches require, by construction, that constraints always be resolved by forces. Nonsmooth motion, however, is an intrinsic aspect of many important physical models applied to better understand the complex behaviors exhibited by mechanical collisions, multi-impact systems, shock-wave propagation in granular media, molecular dynamics (using hard-sphere constraints), geophysical phenomena (e.g., earthquakes, iceberg-calving, etc.) as well as other difficult, nonlinear systems. As such, appropriately resolving nonsmooth motion is likewise critical.

\item \textbf{Smooth exit and approach trajectories.}
Consider once again the simple example of a mass-particle traveling smoothly along a boundary subject to a force pushing into the boundary (Figure \ref{fig:traj_ex}  (c)). If the particle encounters a sufficiently large concavity, it must be able to ``break contact'' and leave the surface. Likewise, if a similar mass-particle is on a ballistic trajectory, that approaches the boundary tangentially (Figure \ref{fig:traj_ex}  (d)), it must not bounce back off, but instead maintain a smooth on-surface trajectory. More generally, smooth trajectories, on the admissible set boundary, should not, in any way, be constrained to remain on the boundary, while initially off-boundary trajectories, that encounter the boundary tangentially, should remain smooth. 
\end{enumerate}

\begin{floatingfigure}[r]{74mm}
\vspace{-4mm}
\hspace{-4mm}
\begin{tabular}{cccc}
\includegraphics[width=16mm]{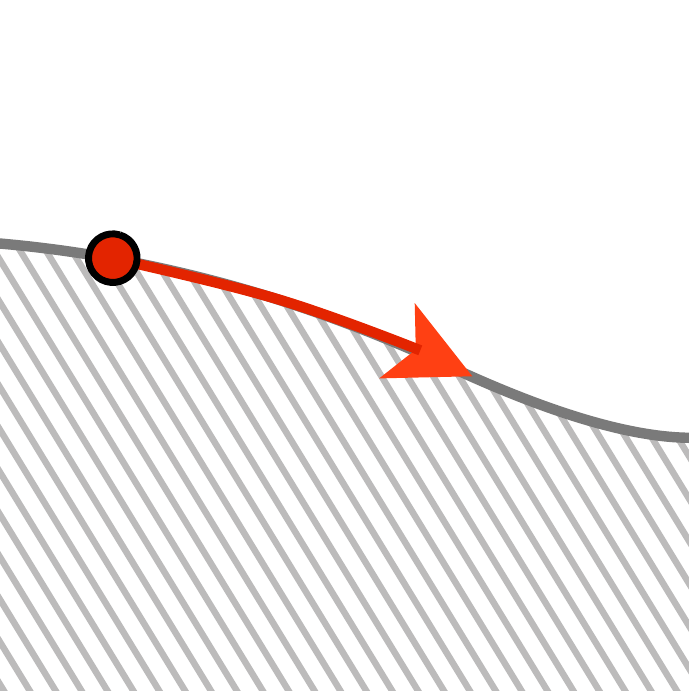}  &
\includegraphics[width=16mm]{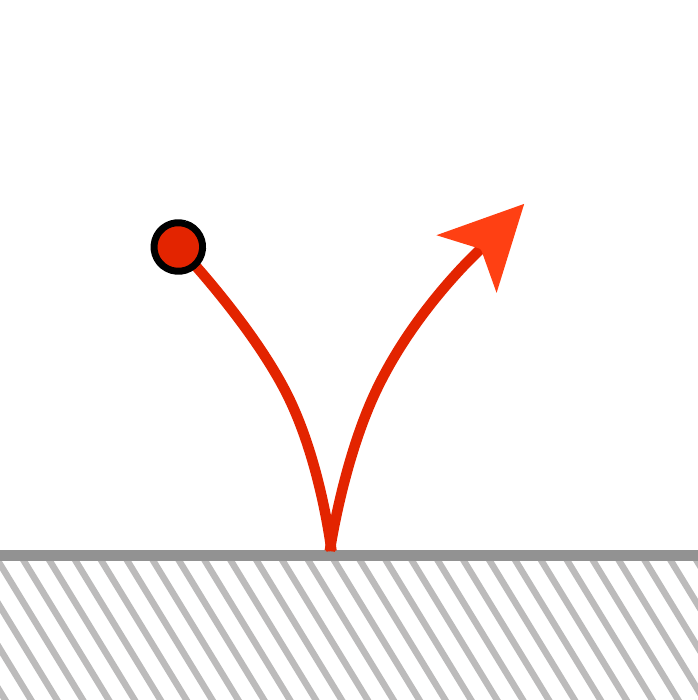}  &
\includegraphics[width=16mm]{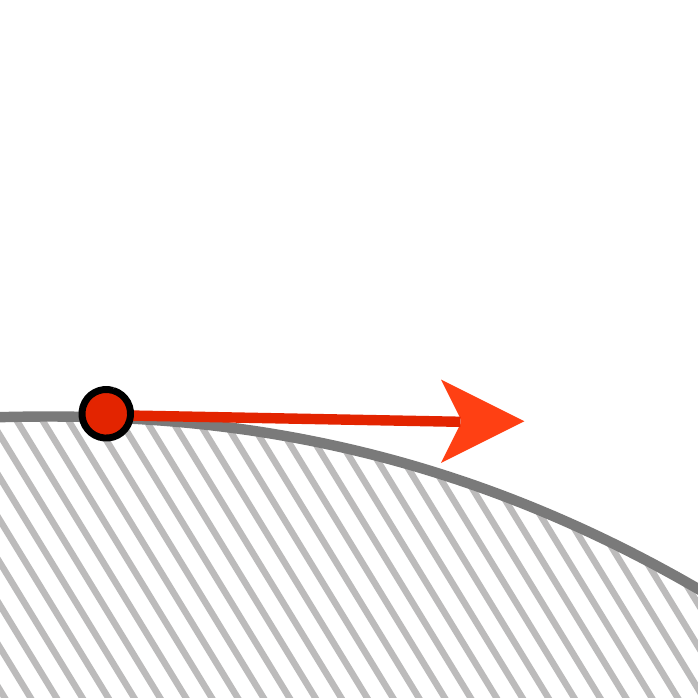} &
\includegraphics[width=16mm]{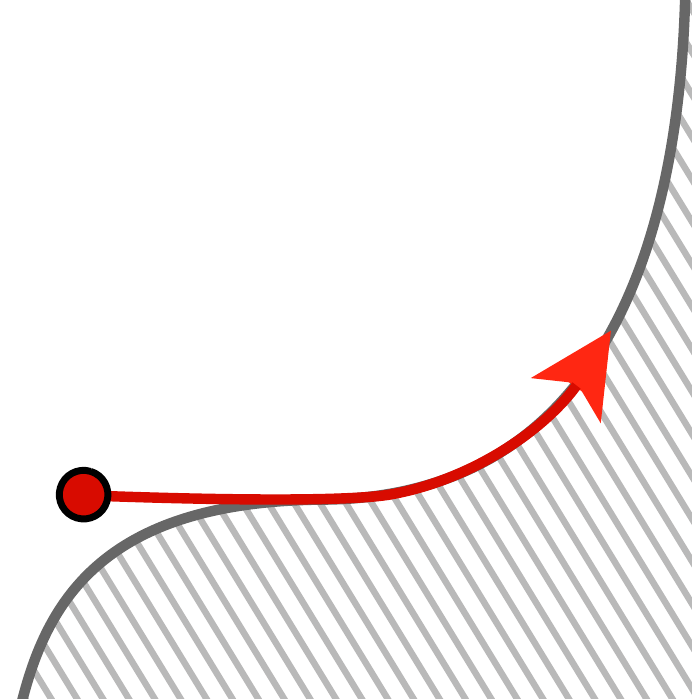}\\
{\small (a)} & {\small (b)}  &  {\small (c)} &  {\small (d)}\\
\end{tabular}
\vspace{-2mm}
\caption{Examples of (a) smooth-boundary, (b) nonsmooth, (c) smooth-exiting and (d) smooth-approaching trajectories. \label{fig:traj_ex}}
\vspace{8mm}
\end{floatingfigure}
%
%
Finally, we note that many fundamental applications of inequality constrained systems (e.g., granular systems, biomechanical locomotion, structural engineering) intrinsically \emph{combine} both smooth and nonsmooth modes and thus require the appropriate resolution of both modes within the same 
framework. 
\vspace{3mm}
\subsection{Overview}
In the following sections we start by discussing related approaches (\S\ref{sec:related}) and then begin to incrementally construct our numerical integration methods, step by step. We first consider the discrete variational perspective (\S\ref{sec:VAR}) and note that a nonsmooth, discrete Hamilton's Principle for unilateral constraints has not been previously considered. To address this gap we propose a Discrete Nonsmooth Hamilton's Principal that leads to a Discrete Euler-Lagrange Inclusion (DELI) formulation (\S\ref{sec:DELI}). We then observe that standard discrete variational structure is not sufficient, on its own, to define a well-posed integrator in this setting (\S\ref{sec:under}). The proposed DELI formulation thus provides a framework for numerical integration in addition to which further structure is required to compose a fully specified integration method.
To instantiate the first such DELI-based method we next consider the structure-preserving behavior of both smooth (\S\ref{sec:smooth_motion}) and nonsmooth (\S\ref{sec:nonsmooth_motion}) time-continuous constrained trajectories. With these observations in place, we then return to the time-discrete problem (\S\ref{sec:gcdvi}), and show that  introducing time-discrete analogues of smooth (\S\ref{sec:discr_smooth}) and nonsmooth (\S\ref{sec:discr_nonsmooth}) boundary structures to DELI allows us to construct a corresponding, symplectic-momentum preserving, smooth-discrete integrator and a momentum preserving, nonsmooth-discrete integrator with observed good long-term energy behavior. Finally, in \S\ref{sec:full}, we derive a discrete, generalized variational integrator (GVI) for generalized inequality-equality constrained systems that resolves both smooth and nonsmooth modes. Further numerical examples illustrating these methods are presented and discussed in \S\ref{sec:ex}.

\subsection{Related Work}
\label{sec:related}

To date, approaches for the structure preserving integration of such inequality constrained, nonsmooth Hamiltonian systems can effectively be divided into two strategies:

\begin{figure}
[h]
\centering
\hspace{-14mm}
\includegraphics[height=0.325\hsize]{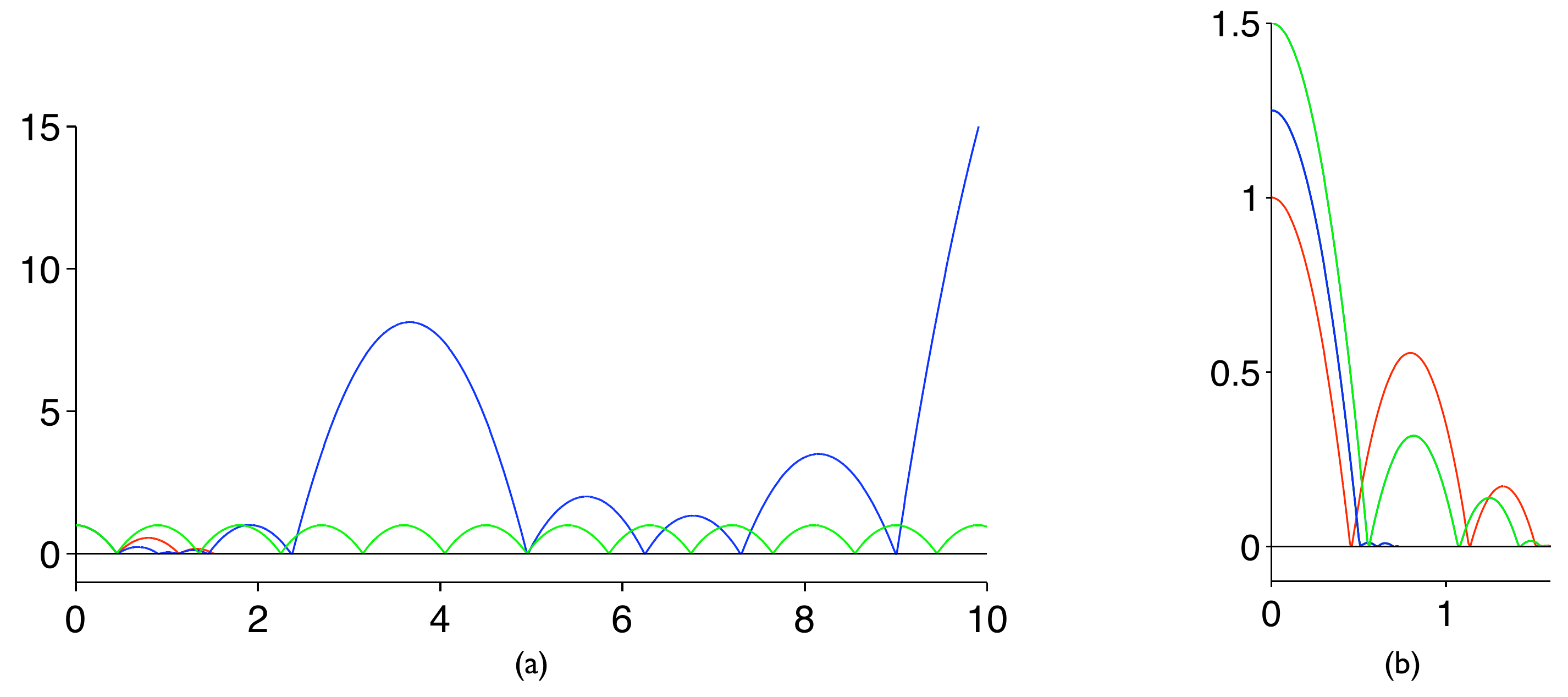}
\caption{{\bf Point-Mass Example:} In these figures we examine the behavior of direct substitution methods on the simple example of a one-dimensional point-mass subject to gravity and a single ground plane constraint at zero. We plot position along the y-axis and time along the x-axis. All examples shown here begin with implicit midpoint as the base, unconstrained numerical method.  In {\bf (a)}, starting with $q(0) =1$, $p(0) =0$, we plot the trajectories of the midpoint-constrained direct method in blue, the endpoint-constrained direct method in red, and
the exact solution's trajectory is plotted in green.
Note, as well, that all 
trajectories
are indentical until the first impact. In {\bf (b)} we plot the the trajectories of the end-point constrained direct method for three small perturbations in initial position. Note that while the endpoint-constrained direct method generates fully dissipative solutions for all cases, the degree of dissipation varies in an energetically inconsistent fashion. \label{fig:1D_particle_gravity}}
\end{figure}

\subsubsection{Direct-Substitution Methods} %
\label{sec:direct_methods}
Energy-momentum and symplectic-momentum methods have been extended to the inequality case by 
the direct substitution of a nonsmooth constraint force~\citep{Laursen97,Kane99,Stewart00,Laursen02,Pandolfi02,Deuflhard07,Khenous08,KrauseWolloth09,Doyen11}. Following the time continuous case discussed above, 
the extended value Hamiltonian for inequality-constrained systems is naturally composed by concatenating the unconstrained 
system's potential with the extended value indicator function. Nonsmooth extensions of standard geometric methods are then composed by direct substitution of the obtained nonsmooth constraint force, $-\partial I_{\vc A} (\vc q)$, as a standard force, handled as dictated by each numerical method.
(See \S \ref{sec:directEx} for examples.) 

Unfortunately, as foreshadowed above, direct substitutions generally destroy energy conservation and generate time-step dependent, nondeterministic behaviors when applied to symplectic methods. In particular, as we prove in Appendix B,
\emph{all direct-substitution, nonsmooth, one-step, symplectic methods can not, in general, preserve, either approximately, nor exactly, the energy of the Hamiltonian system}.
In practice, this results in undesirable time-step-dependent and position-based restitution errors that can produce \emph{both} energy growth \emph{and} dissipation behaviors, while momentum similarly drifts. 

Energy-momentum based direct methods and a variety of extensions and stabilizations are similarly an active and promising area of research~\citep{Doyen11}. Although comparable behaviors to direct symplectic schemes are currently observed in practice~\citep{Doyen11}, drift and stability properties of these methods remain under investigation.

\paragraph{\bf Example: A One Degree of Freedom Particle System}

To illustrate some of these issues consider the following simple example of 
a one dimensional particle with mass, $m=1$, configuration, $\vc q \in \mathbb{R}$, subject to a simple unilateral constraint $\vc g( \vc q) = \vc q \geq 0$, under gravity such that $V(\vc q) = 9.8 \vc q$. 
(Also see Figure \ref{fig:1D_particle_gravity}.)
\citet{Stewart00} uses this simple example to illustrate some of the potential issues involved in applying direct methods to existing symplectic integration methods. Poor energy conservation behaviors are observed for this case by the author for both the (symplectic) implicit midpoint rule~\citep{Hairer03} and the IMEX Newmark scheme from~\citet{Kane99}. In particular, \citet{Stewart00} notes that the energy of these systems is not conserved and that the effective coefficient of restitution varies with state for both methods.

We start by dropping the particle from a height of $q(0) =1$. In Figure \ref{fig:1D_particle_gravity}(a), we plot obtained position for a range of algorithms along the $y$-axis and time along the $x$-axis. In blue we plot  the midpoint-constrained direct method; e.g., using constraint forces of the form $-\partial I_{\vc A} ({\small \frac{\vc q^{t+1} + \vc q^t}{2}})$. This generates, within the same simulation, effective coefficients of restitution that vary widely between dissipation and unstable energy-growth. In general these effects vary the with time-step size and relative position with respect to the constraint boundary. In green we plot the exact solution. Note that both standard collision integration methods (discussed below) and our new method proposed in Section \ref{sec:full} generate trajectories that closely match this solution.

 Finally, we also plot the endpoint-constrained direct method, e.g., using constraint forces of the form $-\partial I_{\vc A} (\vc q^{t+1})$, in red. Note that while the end-point method likewise does not preserve energy, it generates purely dissipative energy errors. For this reason, end-point constraints and/or additional dissipative stabilizations have often been favored for direct methods in the literature~\citep{Kane99,Pandolfi02,Deuflhard07,KrauseWolloth09,Doyen11}. The degree of dissipation, however, varies in a non-physical manner. In particular, as with the mid-point constraint, restitution for the end-point constrained direct method is likewise dependent on time-step size and the relative position of the particle with respect to the constraint boundary. This is illustrated in Figure \ref{fig:1D_particle_gravity}(b), were we plot the the trajectory of the end-point constrained direct method over small changes in initial position. Note that while these three trajectories are all clearly dissipative, the degree of dissipation continues to vary in an energetically inconsistent fashion.

\begin{figure}
[h]
\centering
\begin{tabular}{cc}
\includegraphics[width=0.49\hsize]{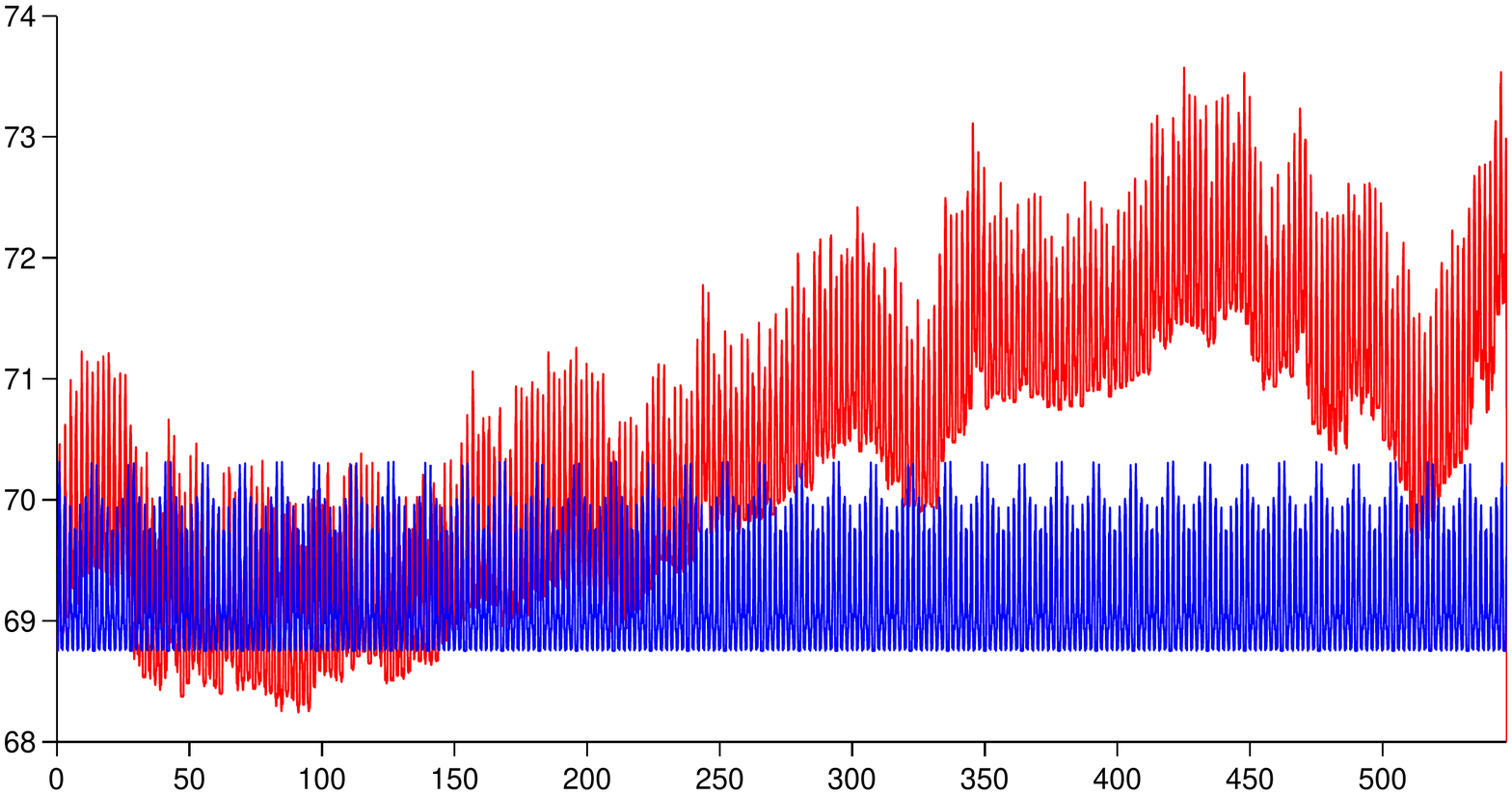} &
\includegraphics[width=0.49\hsize]{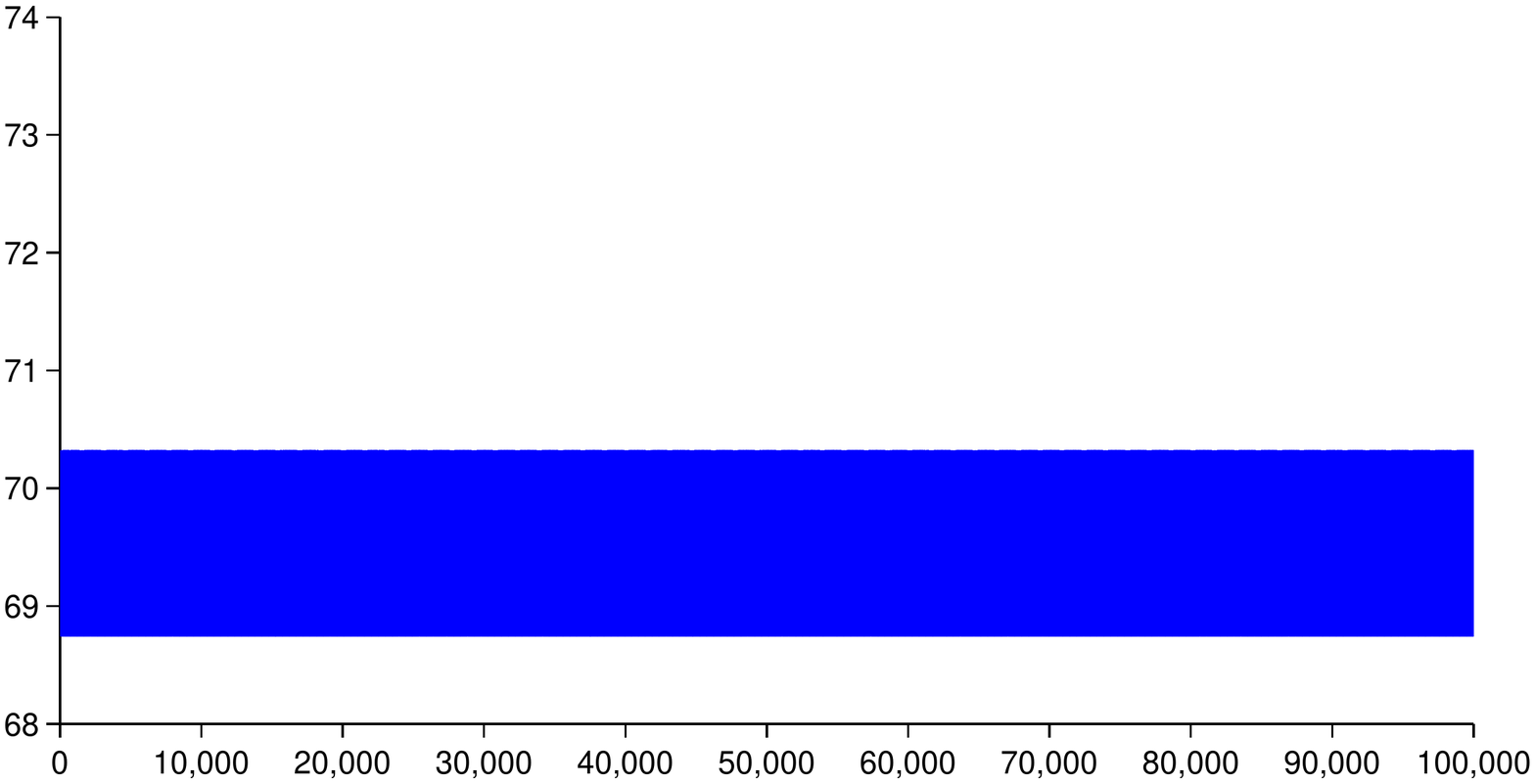}
\figspace
\\ (a) & (b)
\end{tabular}
\caption{{\bf Pogo-Stick Example:} In these figures we plot the energy for a simple, two-dimensional, two-point, mass-spring system vertically oriented, subject to gravity, bouncing on a ground plane constraint. Both methods investigated here begin with St\"{o}rmer-Verlet as the base, unconstrained numerical method. We set $m =1$ for both mass particles. The linear spring stiffness is set to $k=10$, spring length is $l = 5$, and gravity is $g = 9.8213$. In (a) we plot the energy for the first 546 seconds of simulation time, stepping at $h = 10^{-1}$, for both the standard collision integration method in red and our generalized discrete variational integrator (GVI) in blue. After this step the energy in the collision integration blows up. In (b) we continue the GVI energy plot out for 100,000 seconds of simulation time. \label{fig:pogo_Verlet}}
\end{figure}

\subsubsection{Collision Integrators} 
\label{sec:hybrid}

Alternately, \emph{collision integration} methods have also been proposed~\citep{Stratt81,Macneil82,Heyes82,Suh90,Houndonougbo00,Fetecau03,Bond07} in which unconstrained symplectic schemes interrupt fixed-size steps each time a unique constraint is identified as becoming active (i.e., about to be violated). For each such event, an impulsive constraint force (in the form of a reflection) is computed which conserves (depending on the scheme) either a continuous~\citep{Houndonougbo00}, discrete~\citep{Fetecau03} or numerical Hamiltonian~\citep{Bond07}. These methods thus all effectively apply adaptive time-step subdividivision to resolve constraints. Given these variable-step modifications, a trajectory-shadowing, numerical Hamiltonian system can no longer be defined and thus energy conservation can not be expected in the long-term~\citep{Hairer03}.

A novel stabilization scheme, based on the application of two complementary strategies, was recently proposed
to improve
the energetic behavior of collision integrators~\citep{Bond07}. Stabilization is primarily obtained by applying reflections that preserve a second order approximation of the selected integration scheme's (generally St\"{o}rmer-Verlet's) numerical Hamiltonian. A secondary improvement is then also obtained by stepping to and away from constraint events using consistent, higher-order integration schemes. This latter modification, however, is motivated by applications in which constraint events are infrequent. The cost of higher-order integration is thus ameliorated by the assumption that the majority of time-steps are taken constraint free using the base, lower order method. In our development we now consider highly constrained systems where we expect most, if not all, steps to encounter inequality constraints. In this context, any potential stabilization advantage is then given almost exclusively by the higher-order energy reflection. We will examine this further in \S \ref{sec:pogo}.

While these developments can improve the performance of collision integrators, they still demand that only a single constraint be active at any time. Moreover, all collision integrators require expensive root-finding routines to determine each such constraint event. Many applications (e.g., surgical simulation, structural engineering, robotic grasping, granular flow, etc.), however, demand the simultaneous resolution of potentially large numbers of simultaneously active constraints. Individual treatment of constraints in these contexts can lead to ordering bias, unacceptable stability issues, and generally have an especially large detrimental impact in implicit methods.  
Even when such ordering issues are acceptable, (e.g., small time-step, explicit methods) the individual time-ordered evaluation of all active constraint events will be, as \citet{Cirak05} note, computationally intractable for many simulation problems (consider, for example, Zeno's paradox). 

As a generalization it is reasonable, however, to expect that extensions of collision integration methods should be possible, even if not practical, which simultaneously resolve many active constraints. Indeed, one such possible generalization is suggested in \citet{Fetecau03}. Even in this paradigm, however, the interleaving of fixed-step symplectic-momentum methods with a variable-step, energy method still generates drift.

To understand the practical implications of these issues in applying collision integration methods, consider the energy behaviors observed in Figure \ref{fig:pogo_Verlet}(a), above, as well as Figures \ref{fig:pogo_StabVerlet}(a) and \ref{fig:oscillator}(b) and the related discussion in \S \ref{sec:ext_r} and \S \ref{sec:nonlnrOsc}.

\section{Variational Integration and DELI Integrators}
\label{sec:VAR}

Despite the above negative results, nonsmoothly constrained Hamiltonian systems \emph{are} invariant preserving in the usual senses~\citep{Kane99}.
Thus it is still desirable to find geometric integration methods that \emph{do not} introduce numerical artifacts into simulations (dissipative or otherwise). One fruitful direction of research, that we do not pursue here, is the ongoing investigation in extending the discrete \emph{energy-momentum} conserving framework to the nonsmooth setting. Instead we pursue the complementary research program of extending \emph{symplectic-momentum} preserving schemes to resolve inequality constrained dynamics. In particular, observing the natural variational structure implied by inequality constraints~\citep{RockWets98}, we revisit the application of \emph{Discrete Variational Integration}~\citep{MarsdenWest01A} to the inequality constrained setting.  

To date, two approaches towards extending Variational Integrators to the nonsmooth, inequality constrained setting have been considered:

In the first approach, \emph{Nonsmooth Variational Integrators} are generated~\citep{Kane99,Pandolfi02}  by extremizing the standard discrete-action obtained directly from the quadrature of the underlying, smooth, unconstrained system. This obtains a standard Variational Integrator into which a nonsmooth constraint force is then directly substituted. We observe that this approach then obtains nothing other than a direct-substitution symplectic method and is thus covered by Theorem B.1 and subject to the same drift and stability issues as discussed above in \S\ref{sec:direct_methods}. Moreover, we observe that this approach violates the fundamental ``discretize-first, extremize-second'' mantra of discrete variational methodology.

In the second approach, \citet{Fetecau03} observe that nonsmoothness in the variational formulation can be avoided altogether by adding an additional degree-of-freedom that parameterizes time of collision. Composing and extremizing a corresponding discrete variational principle then generates \emph{Variational Collision Integrators} which can be shown to be symplectic~\citep{Fetecau03}. As observed by \citet{Bond07}, however, resulting schemes exactly obtain a standard collision integrator, as discussed above in \S\ref{sec:hybrid}, composed so that each collision preserves a discrete (rather than the continuous or numerical) energy at each reflection.

With these observations in place we then note that, to the best of our knowledge, a \emph{nonsmooth, discrete Variational Principle for  inequality constrained systems remains unexplored}. 
We thus begin our development of nonsmooth geometric methods by developing these initial steps in the fully ``discretize-first'' manner. In particular, we apply Variational Integration (VI) methods~\citep{MarsdenWest01A,Hairer03} as a prescriptive approach for generating a  family of geometric methods for inequality constrained dynamics.

\subsection{VI Methods}
Recall that VI methods, mimicking the continuous derivation of the Lagrangian equations of motion, 
extremize a discrete action on the time interval $t \in [0, T]$, over all possible discrete paths of the form $\{ \vc q^{0} ,...,\vc q^{k} , ..., \vc q^{N} \} $, where $N = T/h$. We begin with the discrete Lagrangian 
\begin{align}
L_d (\vc q^k, \vc q^{k+1}) \simeq \int_k^{k+1} L(\vc q,\dot {\vc q}) dt,
\end{align} 
that gives a quadrature of the natural Lagrangian over a finite interval.

The sum of the discrete Lagrangians, over all intervals, then gives the corresponding standard discrete action, 
\begin{align}
\sum^{N-1}_{k =0} L_d (\vc q^k, \vc q^{k+1}) .
\end{align}
Extremizing this associated discrete action then generates a unique, symplectic-momentum preserving integration scheme~\citep{Hairer03}, with accuracy specified by the order of the quadrature applied~\citep{West04}.

\subsection{DELI}
\label{sec:DELI}
As a point of simple yet, to the best of our knowledge, novel departure we now 
add the nonsmooth penalty term to the discrete action and formulate a discrete, nonsmooth, constrained Hamilton's Principle that enforces hard inequality constraints on all endpoints of the discrete trajectory,
\begin{align}
\delta_k \sum^{N-1}_{k =0} \Big( L_d (\vc q^k, \vc q^{k+1}) -  I_{\vc A} (\vc q^{k+1}) \Big) \ni 0.
\end{align}   
Stationarity then gives us our constrained, Discrete Euler-Lagrange Inclusion (DELI),
\begin{align}
\label{eq:DEL_inclusion}
D_2 L_d ( \vc q^{t-1}, \vc q^t) + D_1L_d ( \vc q^t, \vc q^{t+1}) - \partial I_{\vc A}( \vc q^t) \ni 0,\\
\vc q^{t+1} \in \vc A.
\end{align}   
Given the last two configurations, $\vc q^{t}$ and $ \vc q^{t-1}$,
the DELI integrator advances state to a new configuration, $\vc q^{t+1}$. Here, in this recursive form, we observe that discrete constraint forces, applied to enforce constraints at the end of time-steps, are generated by normal cones constructed at the beginning of time-steps.

We can further clarify this relationship by applying a momentum matching argument~\citep{West04} (or, if preferred, a discrete Legendre-Fenchel transform) to obtain the corresponding discrete momentum variables, $\vc p^t$ and $\vc p^{t+1}$, as well as the discrete phase-space map composed of a position-map inclusion,
\begin{align}
\label{eq:VI_nonsmooth_position}
D_1L_d (\vc q^t, \vc q^{t+1}) + \vc p^t - \partial I_{\vc A}(\vc q^t) &\ni 0,\\
\label{eq:VI_pos_constr}
\vc q^{t+1} &\in \vc A,
\end{align}
and a corresponding momentum-map update equation,
\begin{align}
\label{eq:gcdvi_momentum_map0}
\vc p^{t+1} = D_2 L_d (\vc q^{t}, \vc q^{t+1}).
\end{align}

A one-step formulation of the DELI then follows. Given the discrete phase-space pair, $(\vc p^t, \vc q^t)$, at time $t$, we now apparently can solve the forward position map system,  (\ref{eq:VI_nonsmooth_position}) and (\ref{eq:VI_pos_constr}), for the new constrained configuration $\vc q^{t+1}$. Then, with the new configuration in hand, the discrete momentum update follows directly from  (\ref{eq:gcdvi_momentum_map0}).

\subsection{The Underdetermined Forward Map}
\label{sec:under}

We are now left with an interesting question of causality with respect to constraints. In our above DELI we require that corrective forces, that appear to be generated by the configuration at the beginning of the time-step (i.e, at $\vc q^t $), 
enforce constraints on the configuration at the \emph{end} of the time step. 

In the direct methods we discussed above in Section \ref{sec:direct_methods}, constraint forces are instead 
generated by configurations defined either somewhere in the middle of the time-step, i.e. $\vc q^t + \alpha \vc q^{t+1}, \alpha \in (0,1)$, or else at the end, i.e., $\vc q^{t+1}$. When solving inclusions with these terms, the variational structure of the generalized gradient then \emph{entirely} specifies the constraint force. This, in turn, leads to the destruction of the 
invariant preserving properties of the underlying unconstrained method. See Appendix B,
for further discussion.

To retain the geometric properties of unconstrained methods we thus 
apparently require  some additional degree(s) of freedom in choosing our constraint forces. Due to the odd time-causality noted above, the variational structure of our DELI formulation no longer fully specifies constraint forces and thus provides such 
latitude for their selection. 
To see this, we first observe that in the DELI system merely requires constraint forces to lie in the span of the normal cone generated by the configuration at time $t$,  i.e, $-\partial I_{\vc A}(\vc q^t)$.
Since this cone is independent of final configuration, (\ref{eq:VI_nonsmooth_position}) only specifies the span of possible constraint force directions but does not, unlike in direct-substitution methods, pin down a corresponding force magnitude.

Thus, in the inequality constrained setting, we observe that standard discrete variational structure is not sufficient, on its own, to define a well-posed integrator. We find that constraint forces must be applied along the directions positively spanned by the normal cone and, likewise, all inequalities must be satisfied, but fundamentally, nothing further is specified. In particular, an underdetermined system is composed with any number of solutions, most of which will not generate symplectic-momentum preserving maps. \emph{Our proposed DELI formulation thus provides a framework for numerical integration in addition to which further structure is required to compose a fully specified integration method.} 

In the following sections we will propose the first instantiation of a DELI-based method. We specifically note that the dichotomy between nonsmooth and smooth trajectories is not well-resolved by the discrete variational framework. We will formulate our DELI-based integrator by first considering both the structure of smooth trajectories and the corresponding nonsmooth case. We will then show how imposing discrete analogues of these nonstandard structures upon the DELI formulation composes a fully defined DELI-based integration method and consider the behavior obtained by their application.

\section{Time-Continuous Setting}

In order to consider how to design such methods, we first focus on the ways in which the flow of the constrained Hamiltonian system is transported in the time-continuous setting.

\subsection{Tangency and Normality}
\label{sec:TandN}

As a first step we introduce a few useful notational conventions and then briefly discuss variational structures that allow us to generalize notions of tangency and normality for cases involving multiple active constraints. 

\subsubsection{Notation}

For convenience, throughout the remaining sections we adopt a constraint subset notation. Presuming, as above, that the a full constraint set is indexed by $i \in \{0,...,m\}$, we then define, for each index subset, $\mathbb{K} \subset\{0,...,m\}$, the corresponding constraint subset,
\begin{align}
\vc g_{_{ \scriptsize \mathbb{K}}}( \vc q) \defeq \Big( \vc g_{_{k_0}}( \vc q),..., \vc g_{_{k_l}}( \vc q) \Big)^T, \quad \{k_0,..,k_l \} \equiv \mathbb{K},
\end{align} 
constraint  subset gradient. 
\begin{align}
\vc G_{_{ \scriptsize \mathbb{K}}}( \vc q) \defeq \Big( \nabla \vc g_{_{k_0}}( \vc q)^T,..., \nabla \vc g_{_{k_l}}( \vc q)^T \Big)^T, \quad \{k_0,..,k_l \} \equiv \mathbb{K},
\end{align} 
and the active constraint subset gradient,
\begin{equation}
\vc N_{_{ \scriptsize \mathbb{K}}}(\vc q) \defeq \vc G_{_{ \scriptsize \mathbb{AK}}}( \vc q), \quad \mathbb{AK} = \big\{ i : \>  \> g_{i}(\vc q) = 0, \> \> i \in \mathbb{K} \big\}.
\end{equation}
Consistent with the notation introduced  in Section \ref{sec:problem_statement} for the vector-valued constraint function, $\vc g( \vc q)$,  an absence of subscripting indicates that the entire constraint set is considered.
We then have $\vc G( \vc q) = \vc G_{_{\{0,...,m\}}}( \vc q) $ and $\vc N( \vc q) = \vc N_{_{\{0,...,m\}}}( \vc q)$.

\subsubsection{Normal and Tangent Cones}
\label{sec:nAndt}
The nonsmooth generalization of a normal, the \emph{inward pointing normal cone}~\citep{Rockafellar,Clarke,RockWets98} to $\vc A$ at $\vc q$,  can, in our context, be defined directly with respect to the gradients of the constraint functions, $\vc g$.
It is given, with respect to the full active constraint gradient,  $\vc N(\vc q)$, as $\{ \vc N(\vc q) \> \> \vc x: \vc x \geq 0 \}$. 
Alternately, we note that for $\vc q$ 
in
$\vc A$, the generalized gradient of the indicator function generates the \emph{outward} pointing normal cone to the feasible set. Thus, for all configurations
we have 
\begin{align}
\partial I _{\vc A}(\vc q) = \{ -\vc N(\vc q) \> \> \vc x: \vc x \geq 0 \}.
\end{align}
We note that, in particular, the normal cone reduces, on the interior of $\vc A$, to the singleton, $\{0\}$. Correspondingly, recalling the Euler-Lagrange DI, given in  (\ref{eq:DIEL}), trajectories on the interior of the admissible set remain governed by the standard Euler-Lagrange equations of motion.

Polar to the normal cone, \emph{the tangent cone} to $\vc A$ at $\vc q$, can also be defined with respect to the active constraint gradients, $\vc N$, as
\begin{equation}
\vc T(\vc q) \defeq \{ \vc y : \vc N(\vc q)^T \vc y \geq 0 \}.
\end{equation}
When the admissible set is locally regular, the tangent cone, $\vc T(\vc q)$, encodes the set of feasible directions at $\vc q$, along which infinitesimal motion is locally permissible while, more generally, as we will discuss below, tangent cones generate useful generalizations of tangent spaces to nonsmooth regions of the boundary.

\begin{figure}
[t]
\centering
\includegraphics[height=0.22\hsize]{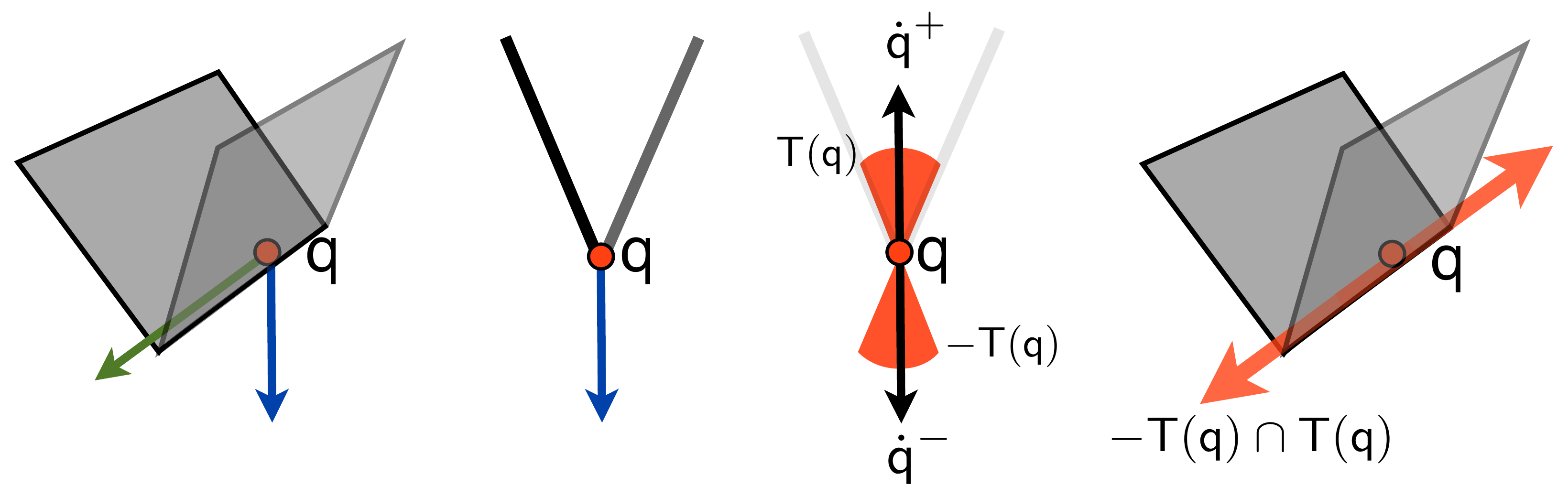} \\
(a) \quad  \quad \quad  \quad    \quad   (b) \quad \quad  \quad   \quad   (c)\quad \quad  \quad   \quad   \quad  \quad  (d) 
\caption{A simple motivating example of smooth and nonsmooth trajectories for constrained systems. See \S \ref{sec:constrained} and \S \ref{sec:smooth_motion} for discussion.
\label{fig:moreauMotivation}}
\end{figure}

\subsection{Constrained Trajectories}
\label{sec:constrained}
In the inequality-constrained setting there is a critical dichotomy between nonsmooth and smooth motion. Consider the simple example of a mass-particle constrained to lie in a nonsmooth valley illustrated in Figure \ref{fig:moreauMotivation} (a).
If the particle lies on the valley floor, then it could have arrived via some motion that is locally tangent to the valley's boundaries, such as moving along the valley floor, e.g., the green arrow in Figure \ref{fig:moreauMotivation} (a). In this case smooth trajectories, such as continuing along the valley floor, exist that can continue to satisfy admissibility. Alternately, the particle could have arrived at the valley floor with a downward pointing velocity, e.g., the blue arrows in Figure \ref{fig:moreauMotivation} (a) and in the cut-away view in Figure \ref{fig:moreauMotivation} (b). In this case there are no possible smooth trajectories that can continue to satisfy admissibility. Some sort of nonsmooth jump at the boundary is then required.

To 
allow for the 
possibility of nonsmooth jumps we consider the left and right limits of velocity,
\begin{align}
 \begin{split}
\dot {\vc q}^-(t) = \lim_{\tau \uparrow t} \dot {\vc q}(\tau) \quad \text{and} \quad 
\dot {\vc q}^+(t) = \lim_{\tau \downarrow t} \dot {\vc q}(\tau),
 \end{split}
\end{align}
where we assume that $\dot {\vc q}$ has bounded variation on finite time intervals; limits on momentum follow similarly. We then consider trajectory intervals and events of interest with respect to the constraint set:
\begin{define}
(Smooth Trajectory Intervals and Nonsmooth Trajectory Events).  We identify {\bfi smooth trajectory intervals} as time intervals $[a,b] \subseteq \mathbb{R}$ where $\dot {\vc q} (t) = \dot {\vc q}^-(t)  = \dot {\vc q}^+(t)$ for all $t \in [a,b]$ and {\bfi nonsmooth trajectory events} as times $t \in \mathbb{R}$ where $\dot {\vc q}^-(t)  \ne \dot {\vc q}^+(t)$.
\end{define}

\subsection{Kinematic Conditions}

Now consider the cut-away view of our point-mass example again, shown in Figure \ref{fig:moreauMotivation} (c) with both the tangent and negated tangent cones depicted in red. If the mass-particle lies on the boundary at time $t$, then it must have approached the boundary (in the left time limit) with a velocity in the negated tangent cone, since no other possible trajectory exists that could have reached the boundary. Likewise, in order to remain an admissible trajectory, the particle must leave the boundary (in the right time limit) with a velocity that is non-negative with respect to all active constraint gradients; thus the right-limit velocity should lie in the positive tangent cone.

This intuition is made rigorous by a Lemma due to~\citet[Proposition 2.2]{Moreau88}, which says that a trajectory lies in the admissible region \emph{for all time} if and only if left velocities always lie in the negated tangent cone and right velocities lie in the positive tangent cone:
\begin{lemma}
\begin{align}
\vc q(t) \in \vc A, \> \> \> \>  \text{for all} \> \> t \> \> \> \> \Longleftrightarrow \> \> \> \>  \dot {\vc q}^-(t)  \in - \vc T(\vc q) \> \> \> \>  \text{and} \> \> \> \>\dot {\vc q}^+(t)  \in  \vc T(\vc q), \> \> \> \> \text{for all} \> \> t.
\end{align}
\end{lemma}

For the nonsmooth case, as in the simple example
discussed above,  
we expect that the left limit will not equal the right and thus we encounter a nonsmooth event requiring an instantaneous jump. We will discuss the time-continuous, nonsmooth case more thoroughly below in Section \ref{sec:nonsmooth_motion}. First, however, we consider the time-continuous, smooth case in greater detail.

\subsection{Smooth Trajectories}
\label{sec:smooth_motion}
We note that smooth motion, with respect to inequality constraints, corresponds at all times to either (unconstrained) motion on the interior of the admissible set, or to any combination of the following boundary cases: (1) encountering constraint boundaries tangentially, (2) leaving constraint boundaries, or (3) tangential motion along constraint boundaries.

More concretely, smooth trajectory intervals are characterized by a \emph{tangent subspace condition}:
\begin{lemma}
If $[a,b]$ is a smooth trajectory interval then  
\begin{align}
\label{eq:TSC}
\dot {\vc q}^-(t) = \dot {\vc q}^+(t) \in -\vc T \big(\vc q(t) \big) \cap \vc T \big(\vc q(t) \big), \> \> \> \> \text{for all} \> \> t \in [a,b].
\end{align}
\end{lemma}
\begin{proof}
Follows directly from Lemma 3.1 and Definition 1.
\end{proof}

 The intersection of the tangent cone and the negated tangent cone thus defines the \emph{tangent subspace} at nonsmooth portions of the admissible set boundary 
 (see our example Figure \ref{fig:moreauMotivation} (d)), 
along which smooth motion is possible and trivially generates the tangent space to smooth regions of the boundary. Conversely, we note that if the left 
velocity lies in the tangent subspace then, locally, an admissible smooth trajectory 
exists.

In particular, while the structure of the normal cone implies that the standard (pointwise) Signorini-Fichera~\citep{Kikuchi88} complementarity\footnote{Note that here and in the following we are using the standard complementarity notation, $\vc x \perp \vc y$, to indicate $x_i y_i = 0$ holds componentwise for the matching entries of $\vc x, \vc y \in \mathbb{R}^{m+1}$. } condition, $\vc g(\vc q) \perp \lambda$, holds everywhere, 
over smooth trajectory intervals Lemma 3.2 and the Euler-Lagrange DI in (\ref{eq:DIEL}) imply that a stronger \emph{constraint-force continuity condition} also holds per constraint: 
\begin{theorem}
Let  $g_i \big(\vc q(a)\big ) = 0$ and  $\nabla g_i\big(\vc q(a)\big)^T \dot{\vc q}^-(a) = 0$; if $\lambda_i(t) > 0$ holds for all $t \in [a,b]$ then
\begin{itemize}
 \item[(a)] $g_i\big(\vc q(t)\big) =  0$,  and 
 \item[(b)] $\nabla g_i\big(\vc q(t)\big)^T \dot{\vc q}^+(t) = 0$,
\end{itemize}
\quad \quad for all $t \in [a,b]$.
\end{theorem}
\begin{proof}
We can equivalently define the inward-pointing normal cone with respect to \emph{all} constraint gradients as $\vc G(\vc q) \lambda = - \partial I_{\vc A} (\vc q)$ with $g_i(\vc q) > 0 \implies \lambda_i = 0$ and $\lambda_i > 0 \implies g_i(\vc q) = 0$. In turn, this implies that if $\lambda_i(t)>0$ on an interval, i.e., $t \in [a,b]$, we must then correspondingly have $g_i(\vc q(t)) = 0$ for all $t \in [a,b]$. The chain rule then gives $\nabla g_i\big(\vc q(t)\big)^T \dot{\vc q}^+(\vc q(t)) = 0$ on the same interval. 
\end{proof}

\subsection{Nonsmooth Motion}
\label{sec:nonsmooth_motion}

With these observations on smooth motion in place, we turn to consider nonsmooth motion.  As discussed above, when new constraints are encountered non-tangentially, i.e., from a normal direction, a nonsmooth trajectory event occurs (see Figure \ref{fig:moreauMotivation} (c)). In the general setting this corresponds to the case where the left limit velocity lies strictly in the interior of the negated tangent cone and requires the satisfaction of \emph{jump conditions}:
\begin{theorem}
If $\dot{\vc q}^-(t) \notin  \vc T(\vc q)$ then  
\begin{itemize}
\item[(a)] $\dot{\vc q}^+(t)  \neq \dot{\vc q}^-(t)$,
\item[(b)] $\vc p^+(t)  = \vc p^-(t) + dp$,
\item[(c)] $dp \in - \partial I_{\vc A} (\vc q)$, and 
\item[(d)]  $ \dot{\vc q}^+(t) \in \vc T(\vc q)$.
\end{itemize}
\end{theorem}
\begin{proof}
By Lemma 3.1 since $\dot{\vc q}^-(t) \notin  \vc T(\vc q)$, $\dot{\vc q}^-(t)$ must lie strictly in the interior of $-\vc T(\vc q)$ and, again by Lemma 3.1, right feasibility satisfaction of $\dot{\vc q}^+(t) \in \vc T(\vc q)$ then requires a nonsmooth trajectory event at time $t$ where $\dot{\vc q}^+(t)  \neq \dot{\vc q}^-(t)$. To satisfy right feasibility, a jump, $dp$, in momentum is required to obtain $\vc p^+(t)  = \vc p^-(t) + dp \in \vc M \vc T(\vc q)$. Finally, at $t$, the Euler-Lagrange DI reduces~\citep{Brogliato99} to the measure-valued, jump inclusion, $dp \in - \partial I_{\vc A} (\vc q)$, which fully determines the feasible set of possible jumps.
\end{proof}

At nonsmooth trajectory events, however, although the measure-valued Euler-Lagrange system, $dp \in - \partial I_{\vc A} (\vc q)$, continues to enforce momentum conservation, it no longer specifies energy behavior. Conservation of energy is then \emph{explicitly} invoked by the side condition
\begin{align}
\label{eq:jump4}
H(\vc q,\vc p^+) = H(\vc q, \vc p^-).
\end{align}

The combined criteria of  
the \emph{jump conditions} and \emph{energy conservation} are satisfied by all conservative, nonsmooth jumps at the boundary of the admissible set. 
In particular, these jump conditions uniquely define an energy and momentum preserving reflection whenever a single constraint is considered. In the general, multi-constraint setting, however, although the jump conditions effectively delimit the space of solutions, an invariant-preserving, constraint-enforcing, impulse is no longer uniquely given~\citep{Glocker04}. In these cases, jumps can effectively be resolved by any of an infinite set of feasible corrective impulses.

We note, however, that a geometric treatment of inequality constrained systems requires a well-posed resolution of jumps with respect to multiple simultaneously active constraints. Methods that satisfy well-posedness by providing uniquely defined, physically motivated, multi-impact solutions for the multi-constraint setting are thus desirable~\citep{ChatRuina98B}. 
In particular, in the following examples, we employ the generalized reflection operator of~\citet{KaufmanPaiGrinspun10}. Alternately, the multi-impact solutions of~\cite{Moreau88} or \citet{ChatRuina98B} can also be applied.

\section{Discrete Setting}
\label{sec:gcdvi}

Given the above analysis, we observe that nonsmoothly constrained Hamiltonian systems correspond to smooth Hamiltonian systems almost everywhere. At instants of transition, where new constraints are encountered or left behind, \emph{either} a smooth trajectory is maintained across the transition \emph{or} a 
jump 
provides a nonsmooth mapping
between two consistent smooth trajectories.
For smooth transitions, motion is given by the Euler-Lagrange DI while, for nonsmooth transitions, jumps instantaneously reflect to a new, energetically consistent, symmetry preserving, smooth Hamiltonian flow. As discussed above in \S\ref{sec:smooth_motion} and \S\ref{sec:nonsmooth_motion}, the conditions specifying the mode of each such transition are entirely given by the relationship between the left limit velocity and the tangent cone. 

To carry these observations into the discrete setting, we will thus first consider compatible, discrete-predicates for distinguishing between smooth and nonsmooth transition modes at each time-step.
With these predicates in place, we will then generate compatible discrete analogues of the time-continuous, smooth and nonsmooth maps.

As a prescriptive guide for constructing these maps, we next observe that, \emph{if} we require our discrete maps to reduce to symplectic steps in the unconstrained case, any obtained \emph{unconstrained} step ending \emph{on} the constraint boundary is always trivially shadowed by an unconstrained numerical Hamiltonian system~\citep{Hairer03}. If the successive \emph{constrained} time-step is \emph{discrete-smooth}, then we should require the corresponding discrete-smooth map to be \emph{symplectic}, so that it is likewise consistently shadowed.  Similarly, if a successive constrained time-step is \emph{discrete-nonsmooth}, then the corresponding discrete-nonsmooth jump-map should be constructed so that it transitions to a new, uniquely defined, invariant preserving numerical Hamiltonian system that is similarly consistent. Following this guide, we conjecture that a piecewise-smooth, numerical Hamiltonian system, for the full inequality constrained system, can then be constructed by gluing together the corresponding smooth numerical Hamiltonians over successive smooth and nonsmooth transitions.

Below we will employ the preceding time-continuous analysis to show one way in which such discrete maps can be generated. The resulting integrator will, by construction,  attempt to preserve correspondence to such a composite, shadowing Hamiltonian system. We note that one fundamental requirement of maintaining such a correspondence is that all new constraint boundaries be considered, in the discrete setting, only at the end (and/or beginning) of each time-step. Indeed, it is specifically the explicit violation of this particular criterion that causes collision integrators to perform so poorly in the nonsmooth case\footnote{Of course, by construction, collision integrators are explicitly designed only for the nonsmooth case.}. 
With these observations in place we now begin the development of our proposed, DELI-based, nonsmooth geometric methods.

\section{Discrete Smooth Setting}
\label{sec:discr_smooth}

In the discrete setting, smooth-trajectory intervals are determined by considering whether the discrete-analogue of the \emph{tangent subspace condition} in (\ref{eq:TSC}) is satisfied at interval endpoints along the discrete path such that
\begin{align}
\label{eq:discrete_smoothness}
\vc M^{-1} \vc p^{t} \in -\vc T(\vc q^{t}) \cap \vc T(\vc q^{t}).
\end{align}
We then identify each DELI forward-map with a discrete-smooth interval on the discrete path. 
Making all constraint functions and their gradients explicit in (\ref{eq:discrete_smoothness}) we find that the discrete left-limit 
\begin{align}
\label{eq:discrete_left_smoothness}
g_i(\vc q^t) = 0 \Longleftrightarrow \nabla g_i(\vc q^t)^T \vc M^{-1} \vc p^t = 0,
\end{align}
and discrete right-limit 
\begin{align}
\label{eq:discrete_right_smoothness}
g_i(\vc q^{t+1}) = 0  &\Longleftrightarrow \nabla g_i(\vc q^{t+1})^T \vc M^{-1} \vc p^{t+1} = 0,
\end{align}
subspace tangency conditions simply require momentum to lie in the cotangent spaces of all active constraints at the beginning and end of all smooth intervals.

The discrete-analogue of the \emph{constraint-force continuity condition} in Theorem 3.3 is then similarly given by 
\begin{align}
\label{eq:discrete_CFC}
\lambda_i > 0 &\Longrightarrow g_i(\vc q^{t+1}) = 0.
\end{align}

\subsection{Discrete-Smooth Integrator}

Reconsidering our derived DELI in (\ref{eq:VI_nonsmooth_position}) we first make constraint functions and their corresponding gradients explicit to obtain the equivalent forward map system,
\begin{align}
\begin{split}
\label{eq:DVI_nonsmooth}
 D_1L_d (\vc q^t, \vc q^{t+1}) + \vc p^t + \vc N(\vc q^t) \lambda &= 0, \quad \lambda \geq 0,\\
\vc g(\vc q^{t+1}) &\geq 0,\\
\vc p^{t+1} &= D_2 L_d (\vc q^{t}, \vc q^{t+1}).
\end{split}
\end{align} 

We then augment DELI with discrete left-limit tangency as a precondition, discrete right-limit tangency as a postcondition and discrete constraint-force continuity over each smooth interval to obtain a DELI-based, \emph{Discrete-Smooth Integrator}:
\begin{align}
\label{eq:ds1}
 D_1L_d (\vc q^t, \vc q^{t+1}) + \vc p^t + \vc N(\vc q^t) \lambda &= 0,\\
 \label{eq:ds2}
0 \leq \lambda \perp \> \vc g(\vc q^{t+1}) &\geq 0,\\
 \label{eq:ds3}
\vc p^{t+1} &= D_2 L_d (\vc q^{t}, \vc q^{t+1}) + \vc N(\vc q^{t+1}) \> \> \mu,\\
 \label{eq:ds4}
\vc N(\vc q^{t+1})^T \vc M^{-1} \vc p^{t+1} &= 0.
\end{align} 
An update step for the Discrete-Smooth Integrator is then generated by first solving the position update equations, (\ref{eq:ds1}) and (\ref{eq:ds2}), for $\vc q^{t+1}$ and $\lambda$. A momentum update then follows from solving equations (\ref{eq:ds3}) and (\ref{eq:ds4}), for $\vc p^{t+1}$ and $\mu$. 

We then observe that
\begin{theorem}
The Discrete Smooth Integrator given in (\ref{eq:ds1}) -- (\ref{eq:ds4}) is symplectic-momentum conserving. 
\end{theorem}
{\it{Proof.}}
See \S\ref{sec:last}.
\\
By construction, the method likewise preserves the subspace tangency of discrete smooth trajectories when moving along, leaving and approaching constraint boundaries. Likewise, the integrator trivially reduces to the unconstrained symplectic method, given by the discrete Lagrangian, on the interior of the admissible set.

\subsection{Spring and Sphere Example}
\label{sec:smooth_ex}
\begin{figure}
[t]
\centering
\begin{tabular}{ccc}
\includegraphics[width=0.15\hsize]{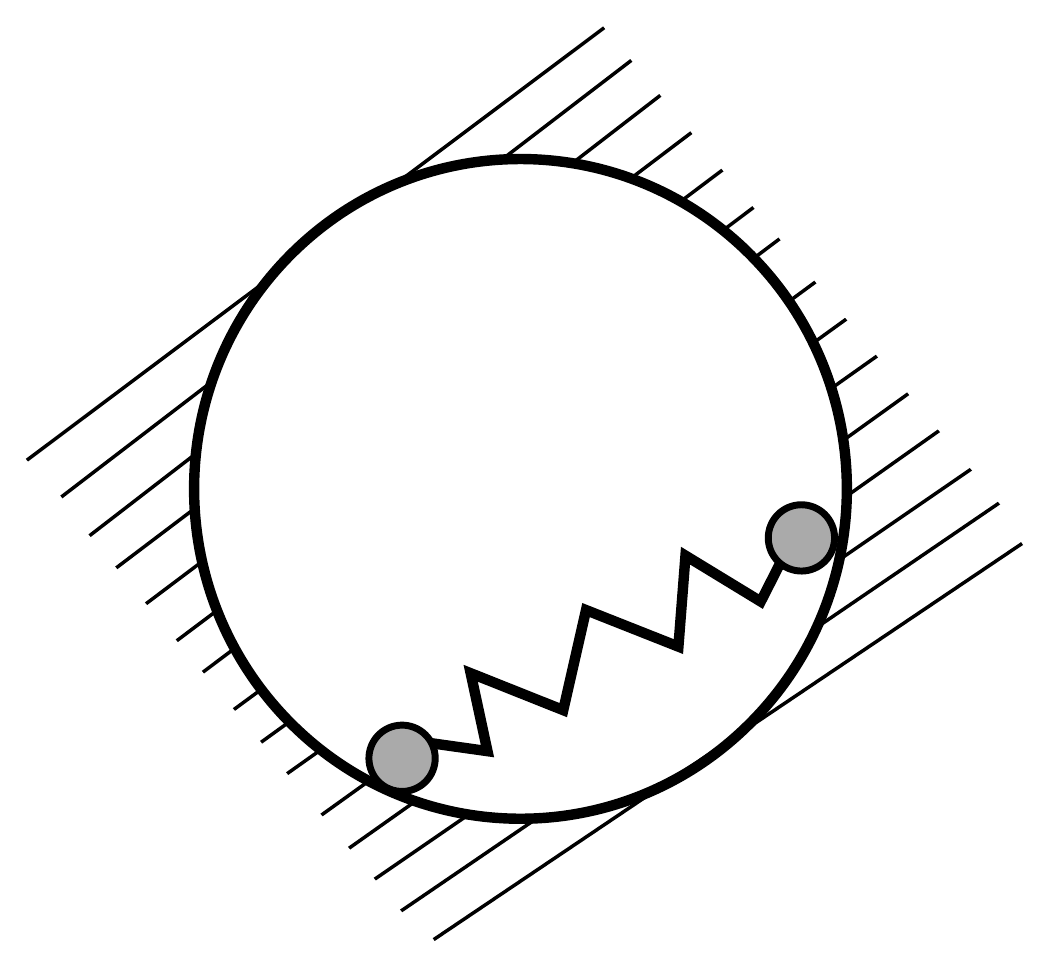}  &
\includegraphics[width=0.285\hsize]{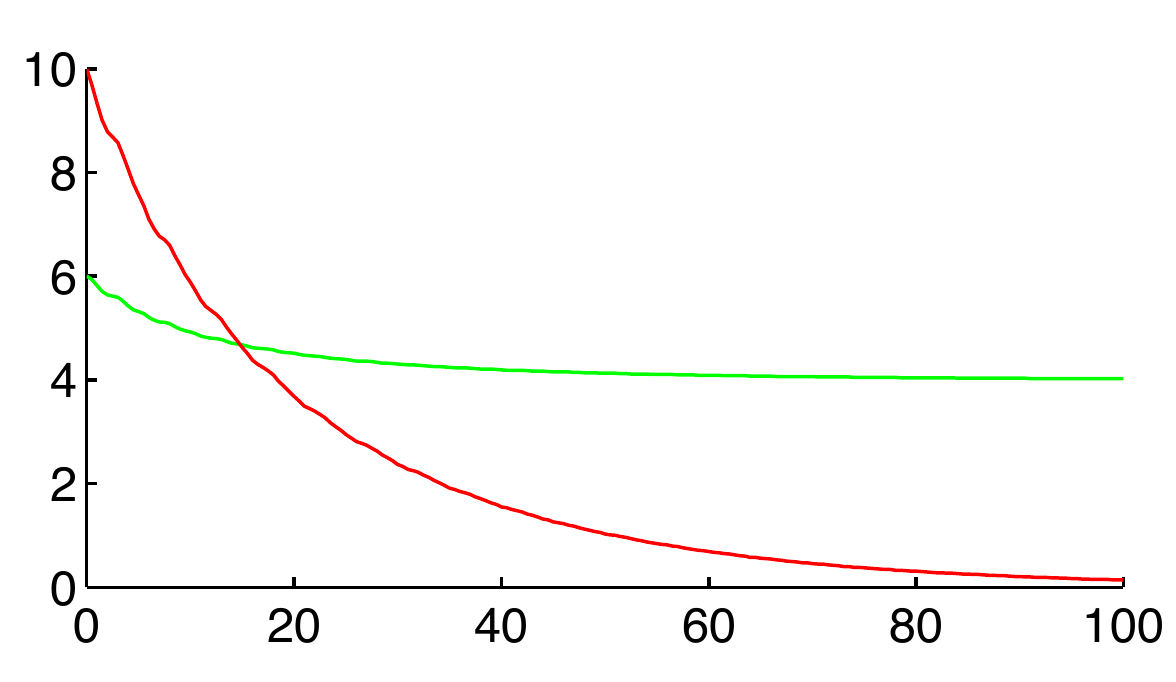}  &
\includegraphics[width=0.535\hsize]{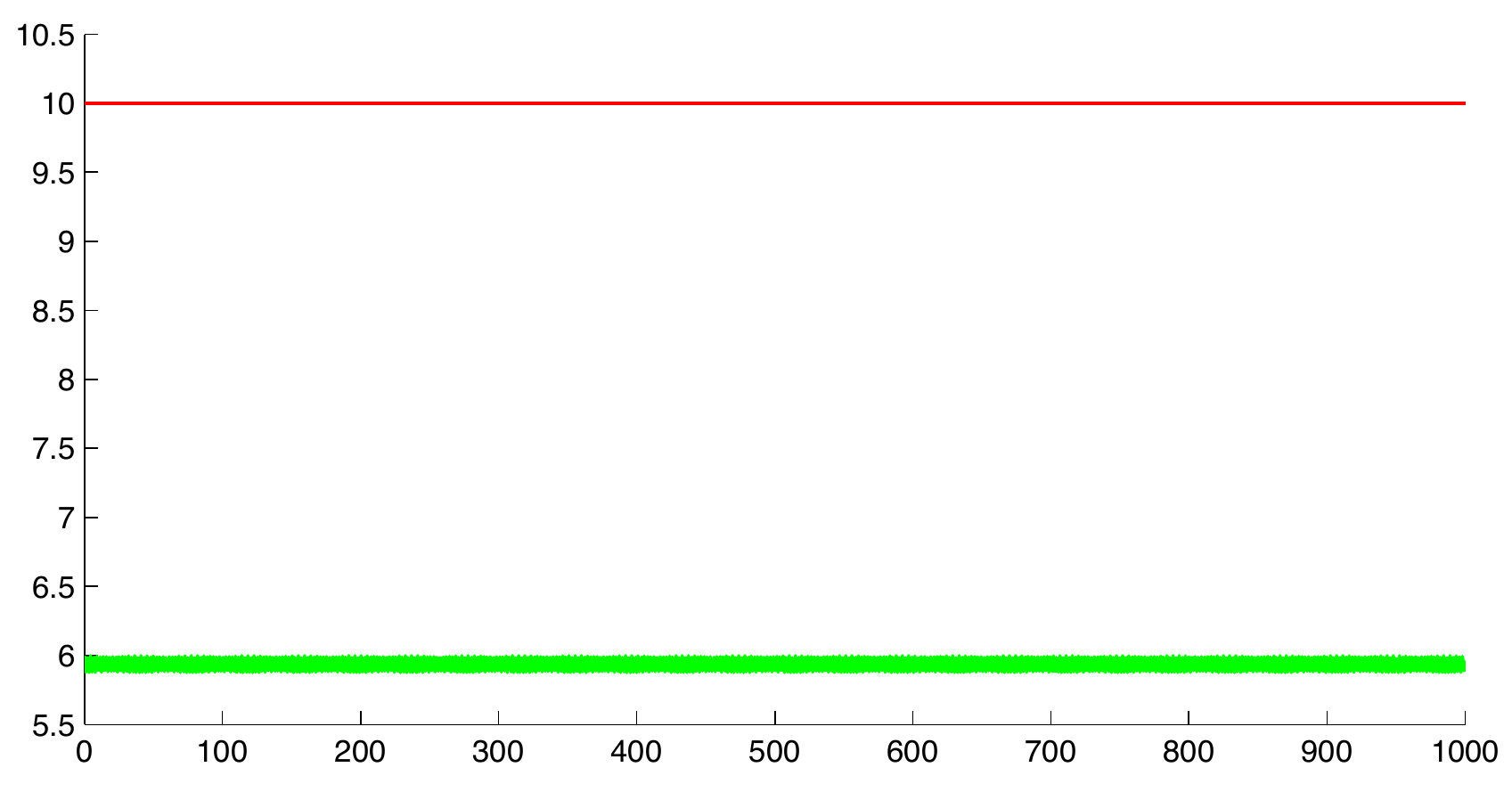} 
\\ (a) & (b) & (c)
\end{tabular}
\vspace{-2mm}
\caption{ 
As illustrated in (a), we consider a smooth trajectory of a constrained system composed of a simple, four-dimensional, two-point, mass-spring system, subject to a circle inequality constraint and an outward-pulling, radial potential. We initialize both particles to a rotational trajectory, initially tangent to the boundary and note that all constraints and potentials, in this example, are rotationally invariant and thus angular momentum should be preserved. Both methods investigated here begin with implicit-midpoint as the base, unconstrained numerical method and are stepped at $h = 5 \times 10^{-1}$. In (b) we plot the energy (in green) and the angular momentum (in red) of the direct, end-point constrained method for the first 100 units of simulation time. We note the decay in both angular momentum and energy. In (c) we similarly plot the (approximate) preservation of energy (in green) and (exact) conservation of angular momentum (in red) obtained by our Discrete Smooth Integrator out to 1000 units of simulation time. See \S \ref{sec:smooth_ex} below for further discussion.  \label{fig:smooth_sphere}}
\end{figure} 

As an initial consideration of the Discrete Smooth Integrator and in comparison to the direct-substitution schemes discussed in \S\ref{sec:direct_methods}, we examine the behavior of a simple, smooth trajectory system. We consider the simple example of a four-degree of freedom, mass-spring system. Configuration is given by the positions of two mass particles,  $\vc q = (\vc q_1,\vc q_2) \in \mathbb{R}^4$, connected by a spring. The system is subject to a standard spring potential between particles and an outward pulling, radial potential, i.e., $V(\vc q) = \sum_{i=1}^2 25/{(\vc q_i^T \vc q_i})  +  \frac{1}{2}(\parallel \vc q_1-\vc q_2 \parallel - l)^2$, with $l = 2\sqrt{2}$. In addition, we apply inequality constraints that require the particles to lie in a sphere, $\vc g_i(\vc q_i) = \parallel \vc q_i \parallel - r \geq 0$, of radius $r = 5$. See Figure \ref{fig:smooth_sphere} (a). Note that all constraints and potentials, in this example, are rotationally invariant and thus angular momentum should be preserved by the system.

We initialize the example system to a smooth, rotational trajectory: both particles start, at time $t=0$, on the sphere boundary, $\vc q_1(0) = (4,-3)$, $\vc q_2(0) = (3,-4)$, with corresponding unit length momenta that point counter-clockwise and tangent to the boundary. This initiates an on-boundary, counter-clockwise rotation of the particles, as well as an oscillatory motion between them.
Both methods investigated here begin with implicit-midpoint as the base, unconstrained numerical method and are stepped at $h = 5 \times 10^{-1}$.

 In Figure \ref{fig:smooth_sphere} (b) we plot the energy (in green) and the angular momentum (in red) of the direct, end-point constrained method for the first 100 units of simulation time. We note that the decay in both angular momentum and energy for the direct method corresponds to the destruction of both the rotational and oscillatory modes of the system. 

In Figure \ref{fig:smooth_sphere} (c) we similarly plot the (approximate) preservation of energy, in green, and the (exact) conservation of angular momentum, in red, obtained by our Discrete Smooth Integrator out to 1000 units of simulation time. Throughout this simulation both particles maintain the smooth, on boundary, rotational and oscillatory modes of the system. 

Finally we note that collision integrators are not suitable for application in such examples since they are, as discussed above, formulated to resolve only nonsmooth boundary interactions. Nevertheless if we do apply a collision integration scheme in this smooth setting (effectively applying a micro-collision model \citep{Brogliato99}), we observe that while energy behavior is not compelling (see \S \ref{sec:pogo} below for a similar example) the trajectory is also qualitatively wrong; the nonsmooth reflections applied by the collision integrator introduce an artificial, nonsmooth, normal, bouncing mode (off the sphere boundary) that increases over time and eventually drives the system to instability.

\section{Discrete Nonmooth Setting}
\label{sec:discr_nonsmooth}

If the discrete left smoothness predicate is \emph{not} satisfied at time $t$, the left limit velocity must lie strictly in the negated tangent cone such that $ \vc M^{-1} \vc p^{t}  \notin  \vc T(\vc q^t)$. A nonsmooth jump is then required to satisfy admissibility. Following our discussion in \S \ref{sec:nonsmooth_motion}, we apply discrete, generalized, energy preserving reflections to resolve these cases. Later, in \S \ref{sec:gr}, we will focus on \emph{how} we derive such discrete generalized reflection operators. First, however, we will consider the critical issue of \emph{when} such reflections 
should be applied.

 \subsection{Extended Reflections}
 \label{sec:ext_r}
We will, in this section, temporarily ignore smooth motion and presume that all encounters with the admissible set boundary are nonsmooth and are always resolved by the application of some energy preserving, generalized reflection (to be defined below in \S \ref{sec:gr}). Presuming that we apply an 
integrator that reduces to a symplectic method in the unconstrained case, standard backwards error analysis~\citep{Hairer03} guarantees that, away from boundaries, there always exists a trajectory shadowing numerical Hamiltonian. In an ideal setting each such unconstrained, symplectic time-step would end either in the interior or \emph{exactly} on the boundary of the admissible set. 

Correspondence to the piecewise-smooth, numerical Hamiltonian, discussed above in \S\ref{sec:gcdvi}, then follows. 
Of course, such nice behavior can not be expected in our generally unsynchronized universe. Instead, we are faced with the likelihood that most, if not all, new constraint boundaries will be encountered midstep. 

The standard approach of existing collision integration schemes (see \S \ref{sec:hybrid}) is to take unconstrained symplectic steps until the admissible boundary is crossed. When a boundary is crossed, collision integrators return 
to the beginning of the time-step and then repeat the forward step using \emph{just} the time-step fraction that lands configuration exactly on the active boundary. A energy-preserving reflection is then performed using current state and then, finally, the integrator completes the remaining portion of the full time-step. 

A simple alternate approach, that, to the best of our knowledge, we introduce for the first time here, is to \emph{always take full time-steps}, using 
augmented or \emph{extended reflections} to incorporate otherwise missed constraint events. 
More concretely, we apply smooth symplectic steps whenever the admissible boundary is not encountered from a normal approach. Anytime current configuration, $\vc q^t$, indicates that a constraint is active and/or a predictor configuration, $\vc q^p$, indicates that a constraint violation is imminent or alternately, depending on the choice of predictor, has already occurred, we perform a generalized reflection at the beginning (equivalently end) of the time step, with respect to both active constraints and \emph{predicted} active constraints. Once the reflection is applied we then take the full smooth, symplectic step. This ensures correspondence between numerical Hamiltonians (as discussed above in \S\ref{sec:gcdvi}), at the cost of inducing errors in the representation of the local constraint boundary geometry. 

\begin{figure}
[t]
\centering
\begin{tabular}{cccc}
\includegraphics[height=0.25\hsize]{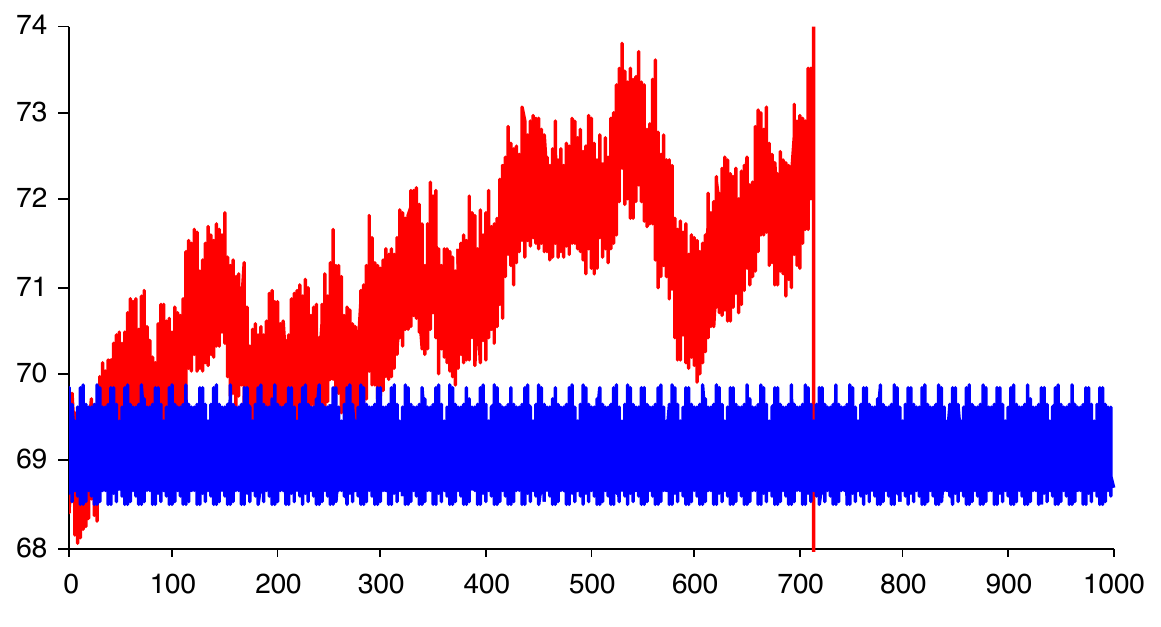}  &
\includegraphics[height=0.2\hsize]{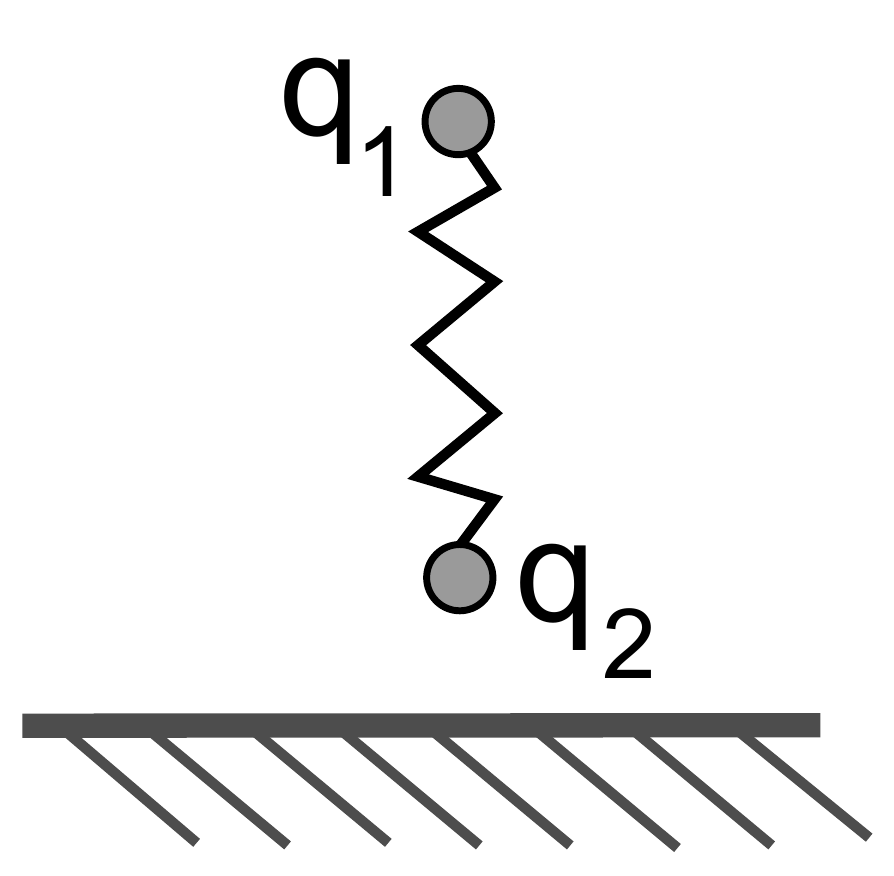}
\\ (a) & (b) 
\end{tabular}
\vspace{-2mm}
\caption{ In (a) we plot the energy for the same simple, two-dimensional, two-point, mass-spring system, shown in (b), and discussed in Figure \ref{fig:pogo_Verlet} and in \S \ref{sec:pogo} below. Here we use exactly the same physical parameters and algorithms as in the original example, with the exception that reflections for both the collision integrator and our extended reflection method are formulated to conserve Verlet's second-order numerical Hamiltonian, rather than the standard continuous Hamiltonian. We plot the energy for the first 1000 units of simulation time for both the collision integrator, in red, and our extended reflection scheme, in blue. Note that, after $t=713$, the energy in the collision integration system blows up. 
\label{fig:pogo_StabVerlet}}
\end{figure}

\subsubsection{Pogo-Stick Example} 
\label{sec:pogo}
To illustrate our new approach and to compare it with collision integration schemes, we consider the simple example of a two-degree of freedom, mass-spring system, subject to a single ground constraint. See Figure \ref{fig:pogo_StabVerlet} (b). We apply St\"{o}rmer-Verlet as the base, unconstrained numerical method. Configuration is given by the positions of two mass particles, vertically oriented, such that the positions of the top and bottom mass particles are given by $\vc q = (q_1, q_2)^T \in \mathbb{R}^2$. The system is then subject to gravity and a ground-plane constraint on the bottom particle, $\vc g( \vc q) =  q_2 \geq 0$.  We set $m =1$ for both mass particles. The linear spring stiffness is set to $k=10$, spring length is $l = 5$, gravity is $g = 9.8$ and the base, unconstrained step-size is $h = 10^{-1}$. 

In Figure \ref{fig:pogo_Verlet}(a) we plot in red the energy of the standard collision integration approach and in blue we plot the energy of our new extended reflection approach using a simple forward time-step predictor. In this first example both methods apply reflections that conserve the continuous Hamiltonian. 
The collision integration method, however, applies reflections at times of contact with the floor, while, in this simple example, our extended reflection method reduces to simply applying reflections at the beginning of all time-steps that would otherwise result in the bottom mass-particle passing through the ground-plane.

In the first time-period, shown in Figure \ref{fig:pogo_Verlet}(a), energy remains fully bounded for the extended reflection method proposed here and, in particular,  varies in the expected manner of standard symplectic methods. Over the same period the hybrid solution experiences energy growth. After $t = 546$, energy in the collision integrator enters a highly oscillatory regime causing the solution to blow up. In Figure \ref{fig:pogo_Verlet}(b) we continue the energy plot of our extended reflection approach out to 100,000 units of simulation time.

Of course we are not limited to continuous Hamiltonian conserving reflections. We may, for instance, expect an improvement in behavior by applying reflections that more closely conserve the shadowing numerical Hamiltonian. In Figure \ref{fig:pogo_StabVerlet} (a) we plot the energy of exactly the same two simulations discussed above \emph{except}, in this example, we apply reflections that preserve St\"{o}rmer-Verlet's second-order numerical Hamiltonian~\citep{Bond07}. We observe a small improvement in the bounds of the energy envelope for our extended reflection method. Similarly, we note that, although the collision integration scheme still experiences energy growth, the obtained solution does not blow up until later, at $t  = 713$.  Note that we could also consider applying higher order symplectic Gauss methods for unconstrained steps leading to collisions~\citep{Bond07}. Given the persistence of constraint events, however, this is effectively the same as simply switching to a higher-order, base, unconstrained integrator for both methods.

\subsection{Discrete Limit Momenta}

We next observe that applying the constraint force $\vc N(\vc q^t) \lambda$ in   (\ref{eq:DVI_nonsmooth}) with respect to the \emph{next} time increment, $t+1$, is indistinguishable from directly modifying the discrete momentum, $\vc p^t$, at time $t$. We thus decompose the constraint multipliers $\lambda$ into right components, $\lambda^{+} = \lambda^{t^+}$, given by a previous nonsmooth reflection, and left components, $\lambda^- = \lambda^{(t+1)^-}$, given by the current discrete-smooth step.
Applying an energy conserving, generalized reflection operator at the \emph{end} of time-step $t$, we then correspondingly 
obtain discrete analogues of the left and right limit momenta,
\begin{align}
\vc p^{t^-}  &= \vc  p^t,\\  
\vc p^{t^+}  &=  \vc p^t + \vc N(\vc q^t) \lambda^{+}.
\end{align}
Here $\vc N(\vc q^t) \lambda^{+}$ gives the impulse generated by the reflection of $\vc p^{t^-}$.
Consistent with our discussion in \S\ref{sec:ext_r}, the above partition of $\lambda$ requires generalized reflections to be applied exclusively at the beginning (or equivalently end) of each time-step. Given our above analysis, we then include the extended reflection in our definition of left and right discrete momenta. 

At the beginning of each time step, we augment the active set with the set of all constraints that would be activated or violated by a predicted configuration. More concretely, we predefine a consistent \emph{predictor configuration}, $\vc q^p$,  and then define an \emph{extended active set}, 
\begin{align}
\label{eq:extendedActiveSet}
\mathbb{A}(\vc q^t, \vc q^p) \defeq \Big\{ i : \> \> \> g_i (\vc q^p) \leq 0, \> \> \> i \in \{0,..,m\}  \Big\} \cup  \Big\{ j : \> \> \> g_j (\vc q^t) = 0, \> \> \> j \in \{0,..,m\}  \Big\}.
\end{align}
The corresponding \emph{extended} discrete left and right momenta are then given by
\begin{align}
\label{eq:discr_left}
\vc p^{t^-}  &= \vc  p^t,\\  
\label{eq:discr_right}
\vc p^{t^+}  &=  \vc p^t + \vc G_{ \tiny \mathbb{A}} (\vc q^t) \lambda^{+}.
\end{align}

\subsection{Discrete Generalized Reflections}
\label{sec:gr}

To compute the discrete reflection term given in (\ref{eq:discr_right}) above, we require a discrete generalized reflection operator, $\textbf{R}$, of the form
\begin{align}
\vc p^+ = \vc p^- + \vc G_{\mathbb{K}}(\vc q) \lambda = \textbf{R}(\vc q, \vc p^-, \mathbb{K}, E),
\end{align}
that defines a \emph{unique} reflection with respect to an active constraint set $\mathbb{K}$. We require that any such reflection operator satisfy the discrete analogues of the \emph{jump conditions} 
from \S\ref{sec:nonsmooth_motion}, 
given by the discrete kinematic feasibility condition, 
\begin{align}
\label{eq:discr_jump1}
 \vc G_{\mathbb{K}}(\vc q)^T \vc M^{-1} \big( \vc p^- + \vc G_{\mathbb{K}}(\vc q) \lambda \big) \geq 0,
\end{align}
the discrete Euler-Lagrange inclusion condition,
\begin{align}
\label{eq:discr_jump2}
\vc p^+ - \vc p^- =  \vc G_{\mathbb{K}}(\vc q) \lambda, \quad \lambda \geq 0.
\end{align}
and conservation of a specified energy function, $E$, (e.g., the discrete, numerical, or continuous Hamiltonian) such that
\begin{align}
\label{eq:discr_jump3}
E(\vc q,\vc p^+) = E(\vc q, \vc p^-).
\end{align} 

As in the time-continuous case, whenever multiple active constraints are considered, i.e., $|\mathbb{K}| > 1$, this problem is underdetermined. From the infinitely many possible solutions, we select the one obtained by applying a discrete adaptation of the generalized reflection operator of \citet{KaufmanPaiGrinspun10}. This 
operator simultaneously guarantees uniqueness, determinism, symmetry-preservation, and satisfies all jump conditions when multiple constraints are active. In \S\ref{sec:discrete_refl} we briefly derive this discrete extension and provide pseudo-code for the implementation we employ. %

\section{Generalized Discrete Setting}
\label{sec:full}
Returning to the full setting where both discrete-smooth and discrete-nonsmooth modes are possible, we finally consider a fully general DELI-based system. Starting with the DELI system given by (\ref{eq:VI_nonsmooth_position}), (\ref{eq:VI_pos_constr}) and (\ref{eq:gcdvi_momentum_map0}), recall that in \S \ref{sec:discr_smooth} we presumed, \textit{a priori},  the precondition of the left discrete-smoothness criteria to obtain the Discrete-Smooth Integrator given by  (\ref{eq:ds1}) through (\ref{eq:ds4}).  In the general setting, however, at any time a subset of the active constraints may be discrete-smooth, while their complement may be discrete-nonsmooth. In such cases we can expect that momentum will be negative with respect to some constraint gradients, and thus nonsmooth modes will be present. 

To partition out constraint subsets that satisfy the discrete-smooth precondition at time $t$, we define the \emph{smooth constraint set} as the set of discrete-smooth constraints, with respect to the discrete-right time-limit at $t$,
\begin{align}
\label{eq:smoothSet}
\mathbb{S}(\vc q^t, \vc p^{t^+}) \defeq \Big\{ i : \> \> \> \nabla g_i (\vc q^t)^T \vc M^{-1} \vc p^{t^+} = 0 \> \> \> \text{and} \> \> \> g_i(\vc q^t) = 0, \> \> \> i \in \{0,..,m\}  \Big\}.
\end{align}
Then, adding the discrete left and right jump terms given by (\ref{eq:discr_left}) and (\ref{eq:discr_right}), we include reflections and obtain a full, discrete, generalized extension of DELI for inequality-constrained systems,
\begin{align}
\vc p^{t^+}  &=  \vc p^t + \vc G_{ \tiny \mathbb{A}} (\vc q^t) \lambda^{+},\\
 D_1L_d (\vc q^t, \vc q^{t+1}) + \vc p^{t^+}  + \vc N_{ \tiny \mathbb{S}} (\vc q^t) \lambda^- &= 0,\\
0 \leq \lambda^- \perp \> \vc g_{ \tiny \mathbb{S}}(\vc q^{t+1}) &\geq 0,\\
\vc p^{t+1} &= D_2 L_d (\vc q^{t}, \vc q^{t+1}) + \vc N_{ \tiny \mathbb{S}} (\vc q^{t+1}) \> \> \mu,\\
\vc N_{ \tiny \mathbb{S}} (\vc q^{t+1})^T \vc M^{-1} \vc p^{t+1} &= 0.
\end{align}

\subsection{Algorithm}
\label{sec:alg}
With the above developments in place, our proposed discrete, generalized variational integrator (GVI) for fully generalized inequality constrained systems is then given by the pseudocode below in Algorithm 1. A GVI takes as input the current state, the selected energy function and the position predictor. It then computes the GVI step and outputs the next state.

\begin{algorithm}[h]
\caption{${\tt GVI }(\vc q, \vc p,  E, \vc q^p)$  \Comment{input:  $\vc q^t$, $\vc p^t$, energy function, position predictor}}\label{alg:gvi_map} 
\begin{algorithmic}[1]
\State $\vc q^{\tiny old} \leftarrow \vc q$ \Comment{cache current position}
\State $\mathbb{A} \leftarrow \Big\{ i : \> \> \> g_i (\vc q^p) \leq 0, \> \> \> i \in \{0,..,m\}  \Big\} \cup  \Big\{ j : \> \> \> g_j (\vc q) = 0, \> \> \> j \in \{0,..,m\}  \Big\}$ \Comment{compute extended active set}
\State $\vc p \leftarrow \textbf{R}(\vc q, \vc p, \mathbb{A}, E)$ \Comment{apply generalized reflection}
\State $\mathbb{S} \leftarrow  \Big\{ i : \> \> \> \nabla g_i (\vc q)^T \vc M^{-1} \vc p= 0 \>  \> \> \text{and} \> \> \> g_i(\vc q) = 0, \> \> \> i \in \{0,..,m\}  \Big\}. $ \Comment{compute smooth set}
\State $\vc q \leftarrow \Big\{ \vc x :  \> \> \> \> D_1L_d (\vc q, \vc x) + \vc p  + \vc N_{ \tiny \mathbb{S}} (\vc q) \lambda = 0, \> \> \> \>   0 \leq \lambda \perp \> \vc g_{ \tiny \mathbb{S}}(\vc x) \geq 0 \Big\}$ \Comment{apply position update}
\State $\vc p \leftarrow \Big\{ \vc y :  \> \> \> \>  \vc y = D_2 L_d (\vc q^{\tiny old}, \vc q) + \vc N_{ \tiny \mathbb{S}} (\vc q) \> \> \mu,\> \> \> \> \vc N_{ \tiny \mathbb{S}} (\vc q)^T \vc M^{-1} \vc y = 0 \Big\}$  \Comment{apply momentum update}
\State \Return $ \vc q, \vc p$ \Comment{ output:   $\vc q^{t+1}$, $\vc p^{t+1}$}  
\end{algorithmic}
\end{algorithm}

\subsection{Bilateral Constraints} 
\label{sec:eq}
Finally, we now further presume that, in addition to the set of inequality constraints, we also wish to enforce a set of equality constraints, as given by  (\ref{eq:equality}). When we consider equality constraints, we observe that the admissible set is now given by a projection of the inequality constraints onto the manifold defined by the additional equality constraints. In turn, this implies that normal cones are now given by the negative span of the active inequality constraint gradients, \emph{projected down to the cotangent spaces} of the equality-constraint manifold. 

It then remains to note that, in the generalized constraint setting, we treat equality constraints as active at all times. We then can add the equality constraints, $\vc f$, and their corresponding gradients, which we will denote by $\vc F$, to our discrete-smooth formulation, during all time-steps. 

For our GVI methods this requires three basic changes. First, for arbitrary inequality constraint subsets, $\mathbb{K}$, the corresponding active inequality gradient subset must be projected onto the local manifold's cotangent space at $\vc q$. Thus, when equality constraints are applied, we  project\footnote{Here $\cdot^+$ indicates the standard matrix pseudoinverse.} the active constraint gradient and redefine as 
\begin{align}
\vc N_{\mathbb{K}}(\vc q) \defeq  \Big( I - \big(\vc F(\vc q)^T  \vc M^{-1} \vc F(\vc q) \big)^+ \vc F(\vc q) \> \vc F(\vc q)^T  \vc M^{-1} \Big) \> \>  \vc G_{\mathbb{AK}}(\vc q), \quad \mathbb{AK} = \big\{ i : \>  \> g_{i}(\vc q) = 0, \> \> i \in \mathbb{K} \big\}.
\end{align}
Next, we simply modify the two constrained updates in the above pseudo-code, given in Algorithm 1, to include the equality constraints. We then obtain a new generalized equality-inequality position update step (to replace line 5 in the pseudocode):
\begin{align}
\vc q \leftarrow \Big\{ \vc x :  \> \> \> \> D_1L_d (\vc q, \vc x) + \vc p  + \vc N_{ \tiny \mathbb{S}} (\vc q) \lambda + \vc F(\vc q) \nu= 0, \> \> \> \>   0 \leq \lambda \perp \> \vc g_{ \tiny \mathbb{S}}(\vc x) \geq 0, \> \> \> \> \vc f(\vc x) = 0 \Big\}
\end{align}
and a corresponding momentum update step (to replace line 6  in the pseudocode):
\begin{align}
\vc p \leftarrow \Big\{ \vc y :  \> \> \> \>  \vc y = D_2 L_d (\vc q^{\tiny old}, \vc q) + \vc N_{ \tiny \mathbb{S}} (\vc q)  \mu + \vc F(\vc q) \xi ,\> \> \> \> \vc N_{ \tiny \mathbb{S}} (\vc q)^T \vc M^{-1} \vc y = 0, \> \> \> \> \vc F(\vc q)^T \vc M^{-1} \vc y = 0 \Big\}.
\end{align}
Otherwise, the remaining pseudocode and all other previous discussion, derivation and guarantees (smooth-interval symplecticity, momentum conservation, etc.) continue to hold.

\section{Further Numerical Examples}
\label{sec:ex}
To understand and examine the behavior of the algorithms proposed here, we investigate implementations of GVI over a range of additional numerical examples. Although a wide variety of predictors can be employed, in the following examples we exclusively apply the simple, unconstrained, forward predictor,  
\begin{align}
\vc q^p = \{ \vc x : \> \> \>   D_1L_d (\vc q^t, \vc x) +\vc  p^t = 0 \},
\end{align}
and reserve an investigation of alternate predictors for later research.
We implement implicit-midpoint and St\"{o}rmer-Verlet based GVI-schemes using the implicit midpoint quadrature,
\begin{align}
L_d (\vc q^k, \vc q^{k+1} ) = h L( \frac{\vc q^{k+1} + \vc q^k}{2},\frac{\vc q^{k+1} - \vc q^k}{h}), 
\end{align}  
and the Verlet quadrature, 
\begin{align}
L_d (\vc q^k, \vc q^{k+1} ) = \frac{h}{2} \> \>\Big( \> \>  L( \vc q^k,\frac{ \vc q^{k+1} - \vc q^k}{h}) +  L( \vc q^{k+1},\frac{ \vc q^{k+1} - \vc q^k}{h}) \> \> \Big), 
\end{align}
respectively.

\begin{figure}
[t]
\centering
\includegraphics[width=1\hsize]{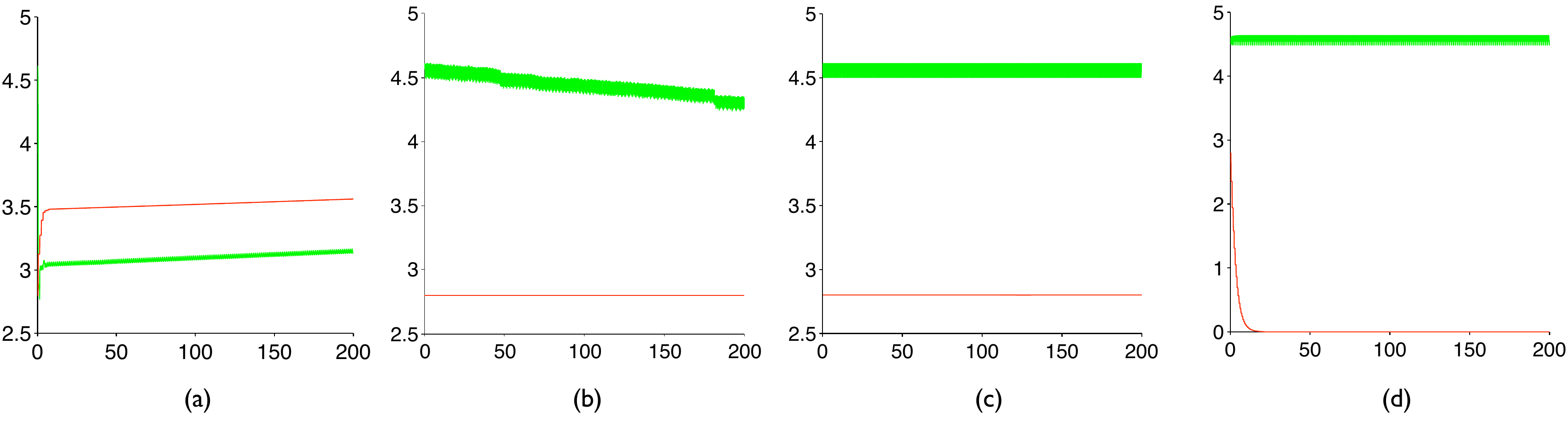}
\caption{{\bf Nonlinear Oscillator Example:} In these figures we examine the behavior of various algorithms on the nonlinear oscillator example in two dimensions.  All methods shown here begin with implicit midpoint as the base, unconstrained numerical method. See Figure \ref{fig:pogo_StabVerlet} (c) and \S\ref{sec:nonlnrOsc} for further  details. Green lines and red lines plot energy and angular momentum respectively. Note that both energy \emph{and} angular momentum should be preserved in this example.  In {\bf (a)} we plot the endpoint-constrained direct method. In {\bf (b)} we plot the standard hybrid method. In {\bf (c)} we plot the GVI integrator, where energy and angular momentum are both preserved long-term. Finally, in {\bf (d)}, to highlight the importance of enforcing nonlinear constraints appropriately, we plot the GVI integrator using linearized constraints. Note that, for this last example, the GVI method continues to preserve energy; however, linearized constraints will no longer preserve rotational symmetry and thus, correspondingly, angular momentum is no longer preserved. 
\label{fig:oscillator}}
\end{figure}

\subsubsection{Nonlinear Oscillator}
\label{sec:nonlnrOsc}
To examine a nonintregable system and momentum conservation properties, we next consider the example of a nonlinear oscillator in  $\mathbb{R}^2$. We start with two spheres in the plane, with respective positions $\vc q_1$ and $\vc q_2$ in $ \mathbb{R}^2$ and radii $r_1$ and $r_2$. The full configuration is then $\vc q = (\vc q_1, \vc q_2)^T$ with a corresponding nonlinear potential for the oscillator 
\begin{align}
V(\vc q) = \sum_{i=1}^2 \parallel \vc q_i \parallel^2 (\parallel \vc q_i \parallel^2 - 1)^2.
\end{align}
See 
the inset figure below.
We then impose the hard-sphere, non-overlap constraint between the two spheres 
using the nonlinear, inequality constraint 
\begin{align}
\label{eq:nonlnrOscConsraint}
\vc g(\vc q) = \parallel \vc q_1 - \vc q_2 \parallel - r_1 - r_2 \geq 0.
\end{align}
Note that both the potential \emph{and} constraint in this example are invariant with respect to rotations and thus angular momentum should be conserved.
%
\begin{floatingfigure}[r]{1in}
\vspace{-.7mm}
\hspace{-8mm}
\includegraphics[width=0.9in]{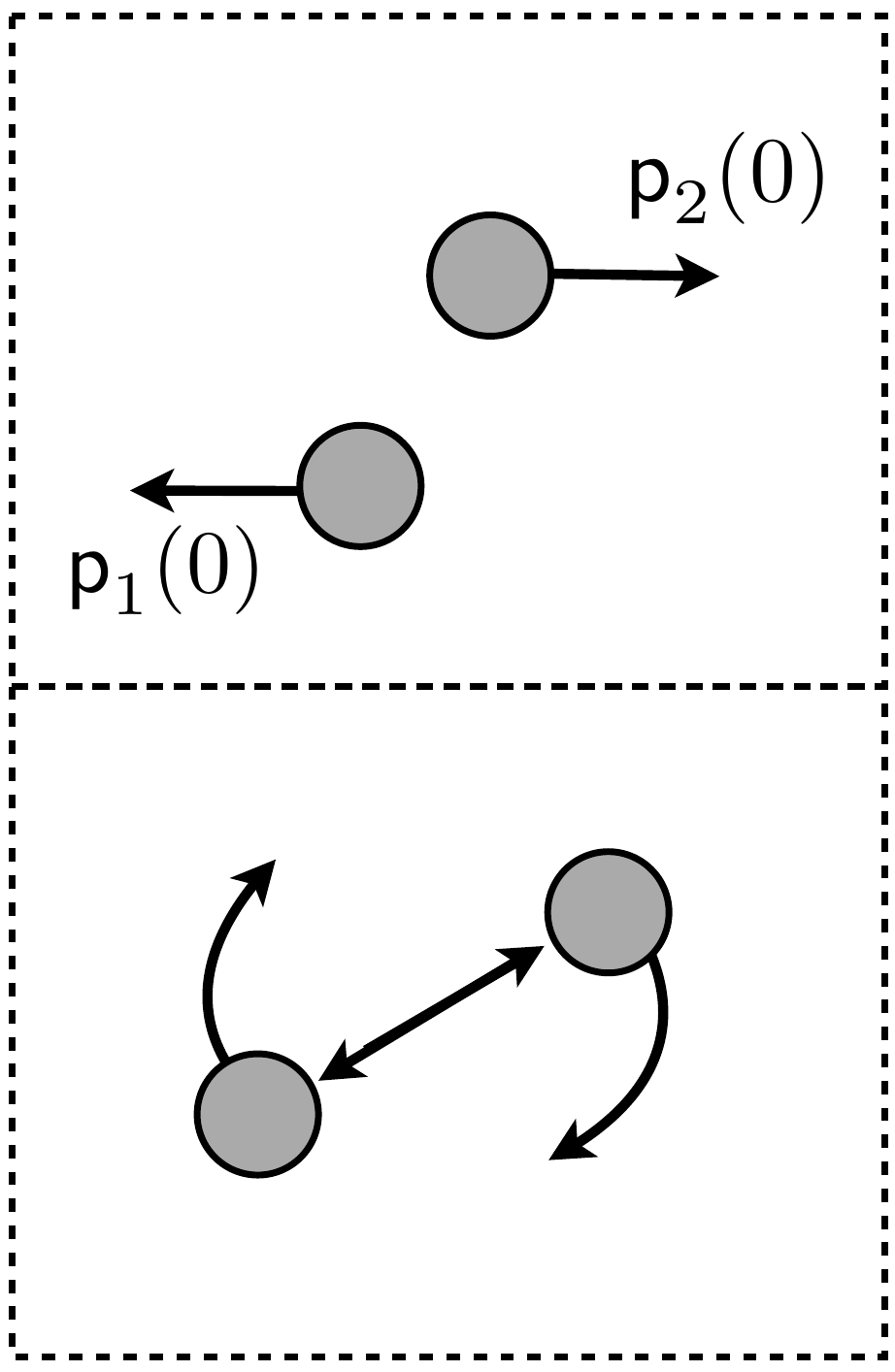}
\vspace{-5mm}
\end{floatingfigure}
%
For this example all methods considered begin with implicit-midpoint as the base, unconstrained numerical method. We then set the step-size to $h =10^{-1}$, mass and sphere radii are set to $m = 1$ and $r = 1$, respectively, for both particles. The starting configuration is given such that the particles are initially separated at the non-equilibrium position, $\vc q_1(0) =  (0, -r-4 \times 10^{-1})^T$, $\vc q_2(0) =  ( 0, r+4 \times 10^{-1})^T$, while the initial momentum initiates a counter-clockwise rotation with $\vc p_1(0) = (1,0)^T$ and $\vc p_2(0) = (-1,0)^T$. The system should then, as illustrated to the right,
enter into a relative equilibrium composed of both periodic bouncing between the particles combined with a global rotation of the full system about the origin.

We compare the results of the endpoint-constrained direct method~\footnote{The midpoint-constrained direct method ( i.e., imposing $\vc g(\frac{\vc q^{t+1} + \vc q^{t}}{2}) \geq 0$) was highly unstable for this example.} (i.e., imposing $\vc g(\vc q^{t+1}) \geq 0$), the standard hybrid method, GVI applied using the full nonlinear constraint given in  (\ref{eq:nonlnrOscConsraint}), and GVI applied using a linearization of the constraints. We include this last example to highlight the potential pitfalls associated with applying constraint linearization since this a popular simplifying strategy in the contact mechanics literature. In Figure \ref{fig:oscillator} we plot the obtained energy in green and angular momentum in red for each method. 

In Figure \ref{fig:oscillator}(a), as might be expected, the endpoint-constrained direct method quickly dissipates all normal modes of oscillation between the two spheres. After this initial dissipation, the spheres enter a relative equilibrium in which they orbit, in constant contact, with a purely smooth motion. This latter mode corresponds exactly to an endpoint-constrained direct method applied to a smooth system. As discussed in Section \ref{sec:TandN}, direct methods applied to smooth motion are likewise not symplectic and thus are still not guaranteed to preserve energy. This is reflected in the plot, where the system continues to generate correspondingly poor energy behavior. In particular, for this latter phase, we see slow but continual energy growth.

In Figure \ref{fig:oscillator}(b) the standard collision integration method likewise displays poor energy behavior. As discussed in Section \ref{sec:hybrid}, energy drift (in the form of dissipation in this case) is generated, in collision integrators, by the interleaving of reflections with variable step-size, unconstrained integration steps. The reflection impulses, however, are applied along constraint gradients evaluated at the time of constraint activation and thus, despite the poor energetic behavior, momentum \emph{is} conserved due to the rotational invariance of the constraint function. 

In Figures \ref{fig:oscillator}(c) and (d) we plot GVI using the true nonlinear constraints and linearized variants respectively. Our GVI method, plotted in Figure \ref{fig:oscillator}(c), enforces the true nonlinear constraints and correspondingly we note that both approximate energy \emph{and} exact angular momentum are conserved long-term, while the combined nonsmooth, rotational \emph{and} oscillatory modes are 
all
preserved.  Finally, we note that, when applied to linearized constraints, as in Figures \ref{fig:oscillator}(d), GVI continues to preserve energy but, due to constraint linearization, rotational invariance no longer holds and thus angular momentum is no longer conserved. This highlights the cost of linearizing constraints when applying geometric methods.

\begin{figure}
[h]
\centering
\begin{tabular}{ccc}
\includegraphics[width=0.15\hsize]{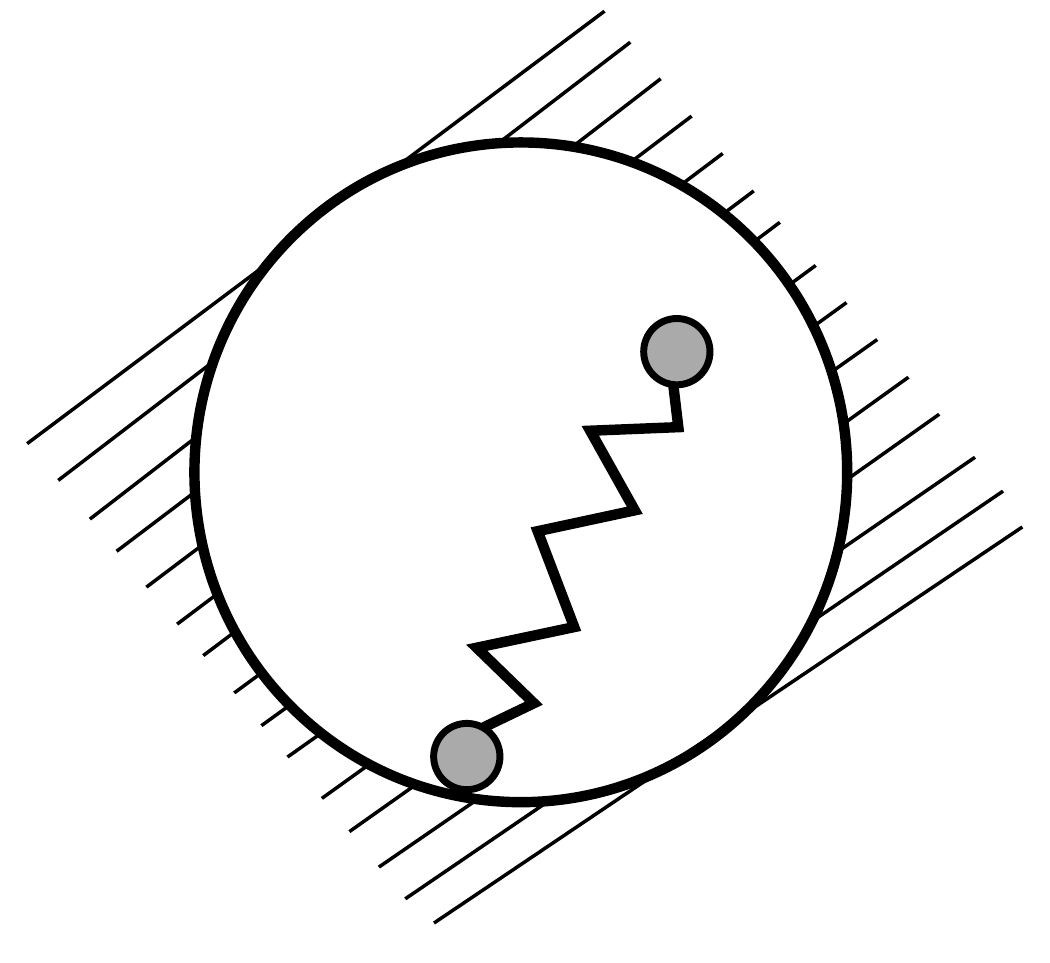} &
\includegraphics[width=0.4\hsize]{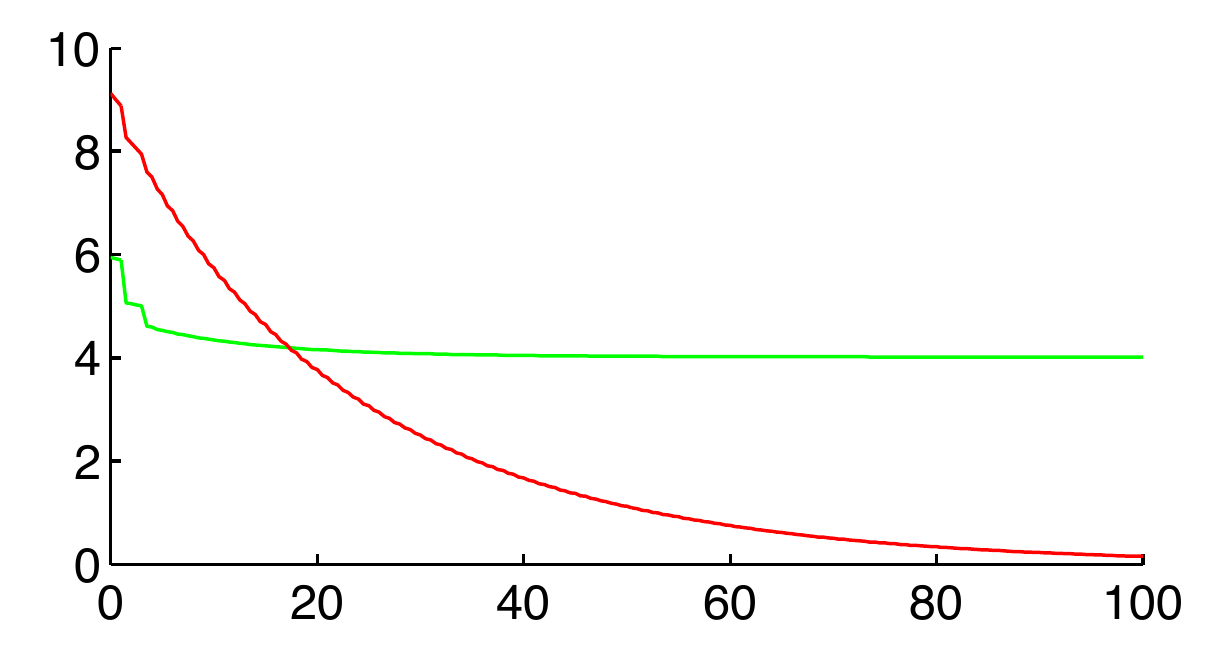}&
\includegraphics[width=0.4\hsize]{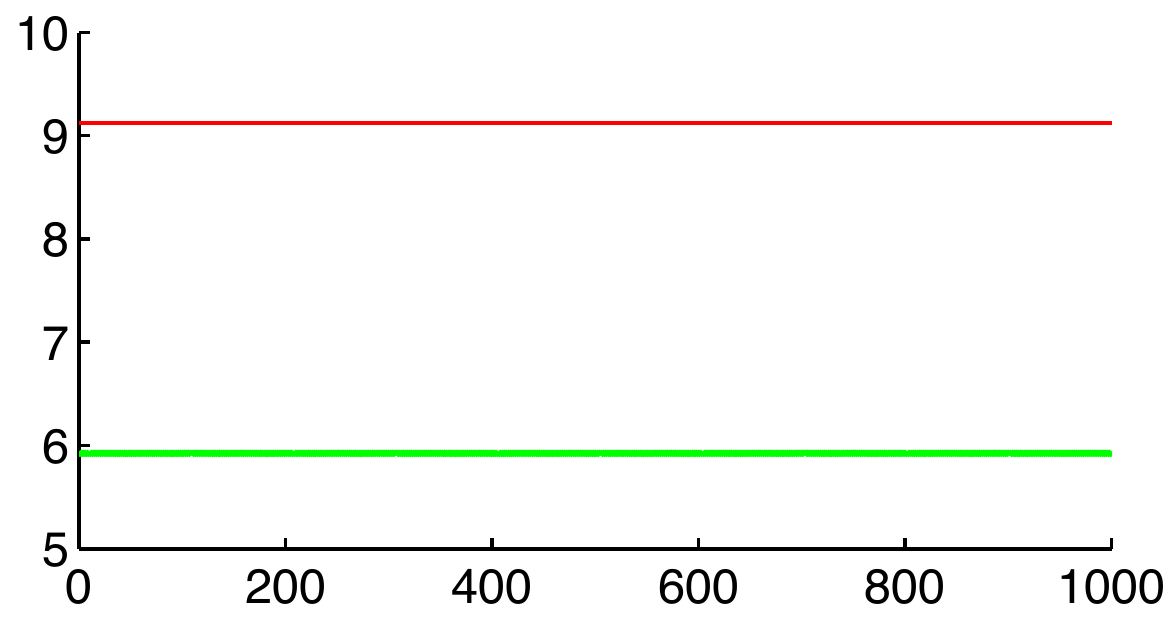}
\\ (a) & (b) & (c) 
\end{tabular}
\caption{
We consider the combined smooth \emph{and} nonsmooth trajectory of a constrained system composed of the same simple constrained mass-spring system considered in \S\ref{sec:smooth_ex} and Figure \ref{fig:smooth_sphere}.  Unlike the previous example, however, we now initialize the first particle to an off sphere boundary position and set the second particle to an on-boundary position as illustrated in (a). As in \S\ref{sec:smooth_ex} both particles are initialized with corresponding unit length momenta that point counter-clockwise and tangent to the boundary. This initiates a counter-clockwise rotation for both particle with an oscillatory motion between them via forces applied by the spring. The first particle, however, should now begin a nonsmooth oscillatory bouncing off the surface, while the trajectory of the second particle should remain smooth and on boundary throughout the simulation. Otherwise all other physical parameters remain unchanged from the last example.
In (b) we plot the energy (in green) and the angular momentum (in red) of the direct-substitution, end-point constrained, method for the first 100 units of simulation time. We note the dissipation in both angular momentum and energy. This includes a rapid  decay of all nonsmooth bouncing modes. In (c) we similarly plot the (approximate) preservation of energy (in green) and (exact) conservation of angular momentum (in red) obtained by GVI out to 1000 units of simulation time. See \S \ref{sec:smoothNonsmooth_ex} for further discussion.  
\label{fig:smoothNonsmooth}}
\end{figure}

\subsection{Smooth and Nonsmooth, Spring and Sphere}
\label{sec:smoothNonsmooth_ex}

As a first consideration of an example with combined smooth \emph{and} nonsmooth boundary modes, we reconsider our spring and sphere example from \S\ref{sec:smooth_ex}. In this example we leave all physical and time-step parameters unchanged from from \S\ref{sec:smooth_ex} \emph{except} that we now initialize the example system to a combined smooth \emph{and} nonsmooth, rotational trajectory. We initialize the first particle to an off sphere boundary position at $\vc q_1(0) = (4,-1)$. We then set the second particle to an on boundary position, $\vc q_2(0) = (3,-4)$. See Figure \ref{fig:smoothNonsmooth}(a). As in \S\ref{sec:smooth_ex} both particles are initialized with corresponding unit length momenta that point counter-clockwise and tangent to the boundary. This initiates a counter-clockwise rotation for both particles with an oscillatory motion between them via forces applied by the spring. The first particle, however, should now begin a nonsmooth oscillatory bouncing off the surface, while the trajectory of the second particle should remain smooth and on boundary throughout the simulation. 

In Figure \ref{fig:smoothNonsmooth} (b) we plot the  energy (in green) and the angular momentum (in red) of the direct, end-point constrained method for the first 100 units of simulation time. We note that decay in both angular momentum and energy for the direct method corresponds to the destruction of both the nonsmooth \emph{and} smooth modes of the particles, as well as the rotational and oscillatory modes of the system. In Figure \ref{fig:smoothNonsmooth} (c) we similarly plot the (approximate) preservation of energy, in green, and the (exact) conservation of angular momentum, in red, obtained by our GVI method out to 1000 units of simulation time. Throughout this simulation, GVI maintains the nonsmooth bouncing mode of the first particle, the smooth-on-boundary mode of the second particle, as well as the rotational and oscillatory modes of the full system.

\subsubsection{Newton's Cradle}

\begin{figure}
[h]
\centering
\begin{tabular}{cc}
\includegraphics[width=0.8\hsize]{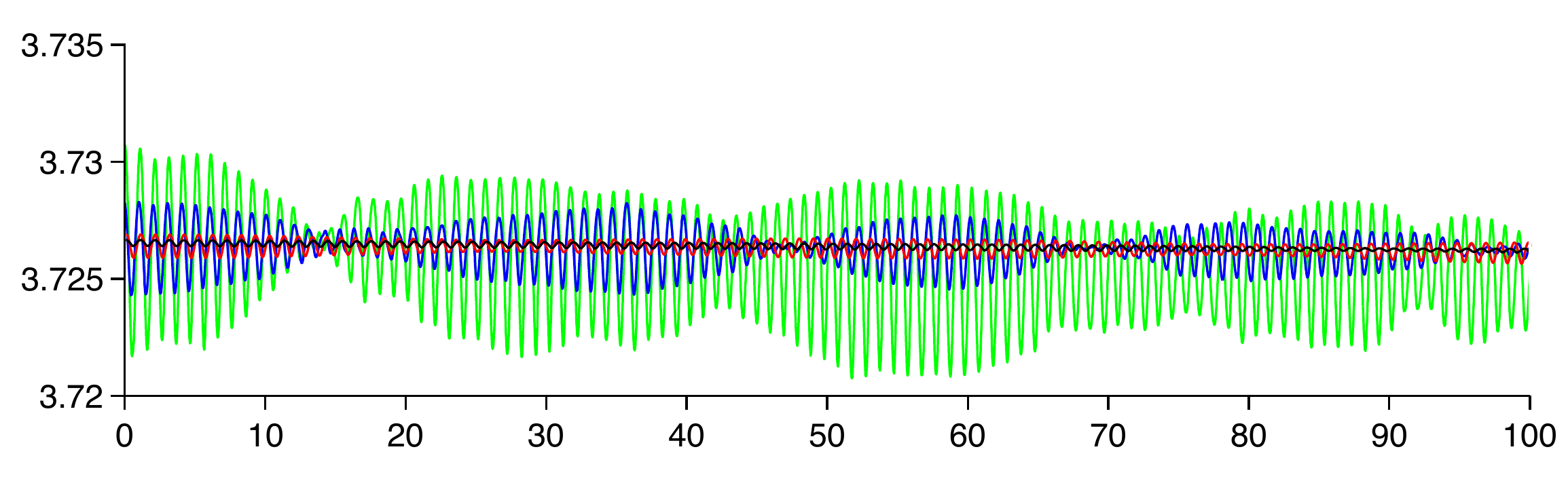} &
\includegraphics[height=0.26\hsize]{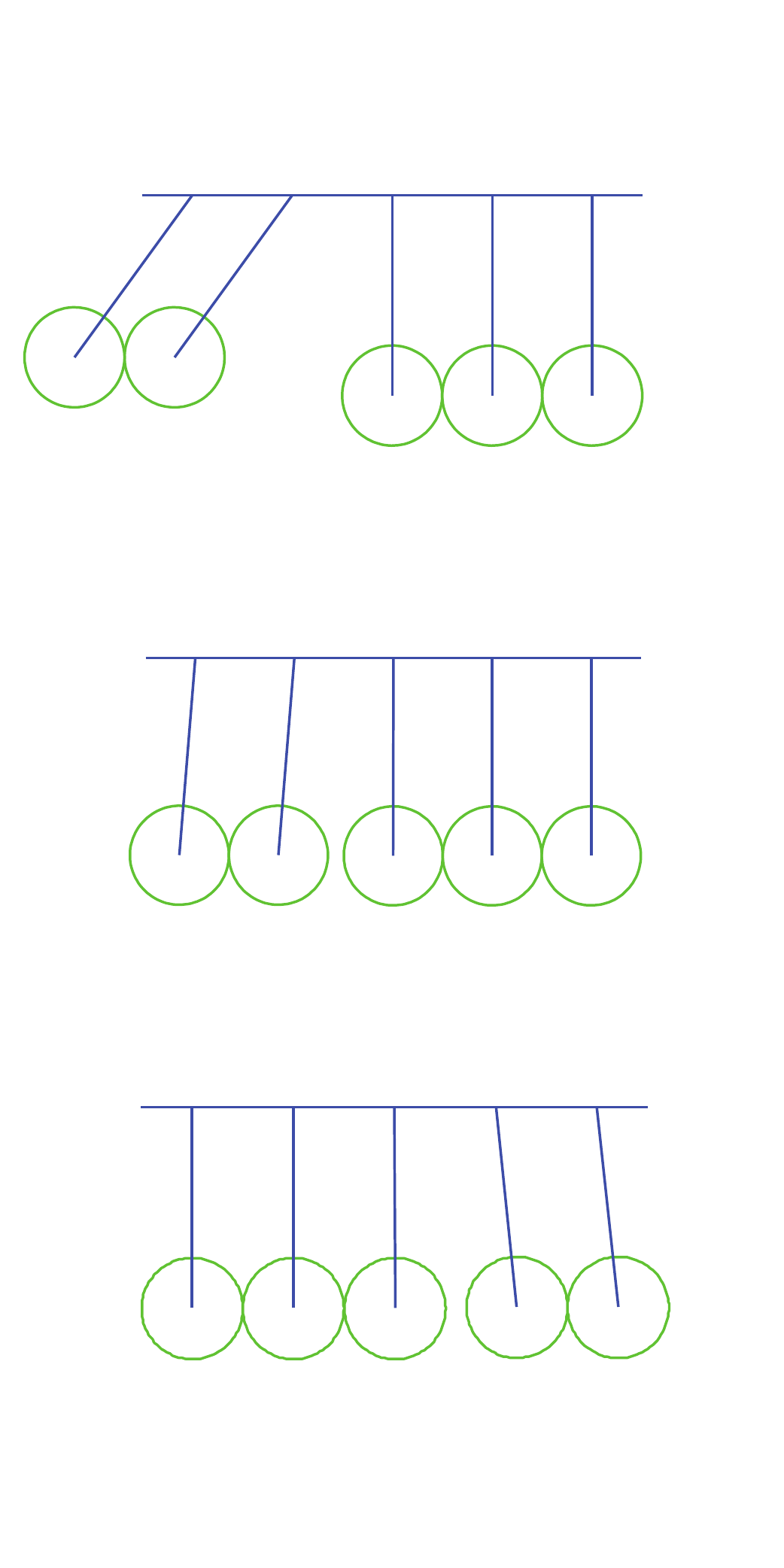}
\\ (a) & (b)
\end{tabular}
\caption{{\bf Newton's Cradle Example:} In (a) we compare the results of our implicit-midpoint-based GVI method applied to a Newton's Cradle System over a range of decreasing step sizes, $h = 3 \times 10^{-2}$ (green), $2 \times 10^{-2}$ (blue),  $10^{-2}$ (red), and $5 \times 10^{-3}$ (black). We note that in the passage from $h = 3 \times 10^{-2}$ to $h = 10^{-2}$, we obtain the classic transition to stable energy behavior that is characteristic of symplectic methods. In (b) we illustrate the behavior of the integrator with a series of snapshots from the simulation. \label{fig:newtons}}
\end{figure}

Next we examine an example with combined smooth \emph{and} nonsmooth boundary modes
\emph{and} the simultaneous enforcement of both equality and inequality constraints. We consider the classic, pendulum-based, Newton's Cradle system composed of $n$ initially touching spheres subject to gravity. Each sphere $i \in [1,n]$ is constrained by a corresponding equality constraint 
\begin{align}
\vc f_i(\vc q) = \parallel \vc q_i - \vc p_i \parallel - \> \> \vr l = 0,
\end{align}
to remain at constant length $\vr l$ from its hanging attachment point $\vc p_i$. Neighboring pairs $i,j$ of spheres are then further constrained by non-overlap constraints 
\begin{align}
\vc g_{i,j}(\vc q) = \parallel \vc q_i - \vc q_j \parallel - r_i - r_j \geq 0.
\end{align}
Thus, for an $n$-sphere cradle we enforce $n$ bilateral and $n - 1$ unilateral constraints and obtain a combination of smooth motion along the constraint boundary interfaces between neighboring spheres in moving and resting clusters, holonomic motion along the manifolds given by the pendulum constraints, and nonsmooth motion at each sphere impact.  

For this example we begin with implicit-midpoint as the base, unconstrained numerical method. 
We then set the masses, sphere radii, and pendulum lengths to $m = 1$, $r = 0.25$ and $\vr l = 1$, respectively, for all spheres. We examine a five-ball cradle. The starting configuration is given such that all but the two leftmost spheres are at rest and in contact. The two leftmost spheres are initially pulled back into a non-equilibrium, contacting position, at an angle of $-\pi/5$ from rest. The system should then enter into the characteristic Newton's cradle behavior as illustrated by the simulation snapshots shown in Figure \ref{fig:newtons} (b).

We compare the results of 
GVI over a range of decreasing step sizes.  
In Figure \ref{fig:newtons} (a) we plot the obtained energy for  $h = 3 \times 10^{-2}$, $2 \times 10^{-2}$,  $10^{-2}$, and $5 \times 10^{-3}$. We note that in the passage from $h = 3 \times 10^{-2}$ to $h = 10^{-2}$ we obtain the classic transition to stable energy behavior that is characteristic of symplectic methods. Furthermore, as step size decreases, the cradle preserves the correct cyclical behavior, expected from Newton's cradle, for increasingly longer periods. Finally, we note that both the direct methods, and the hybrid schemes, along with generating their expected energy errors, do not predict the correct physical behavior of Newton's Cradle.

\begin{figure}
[h]
\centering
\begin{tabular}{ccc}
\includegraphics[height=0.25\hsize]{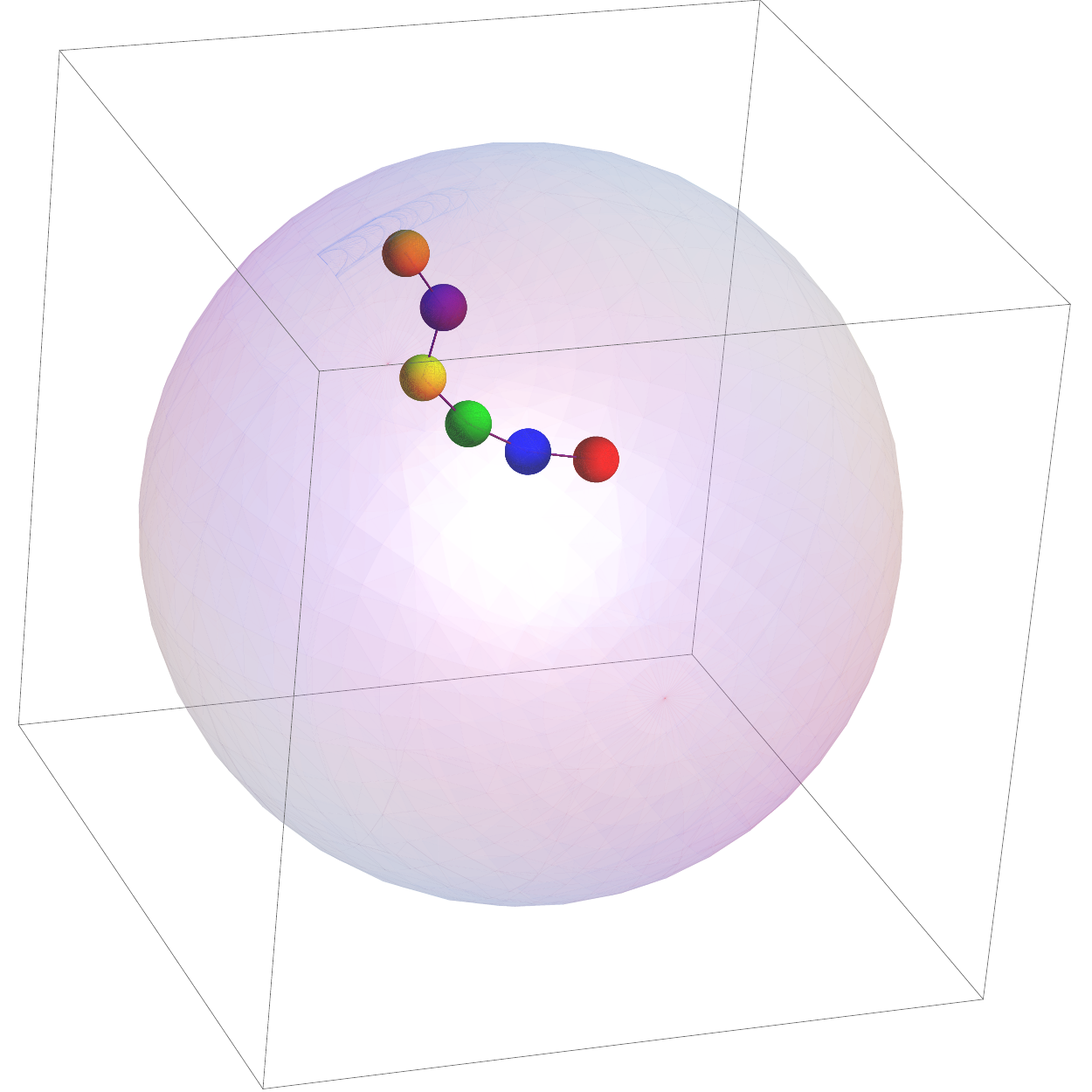} &
\includegraphics[height=0.25\hsize]{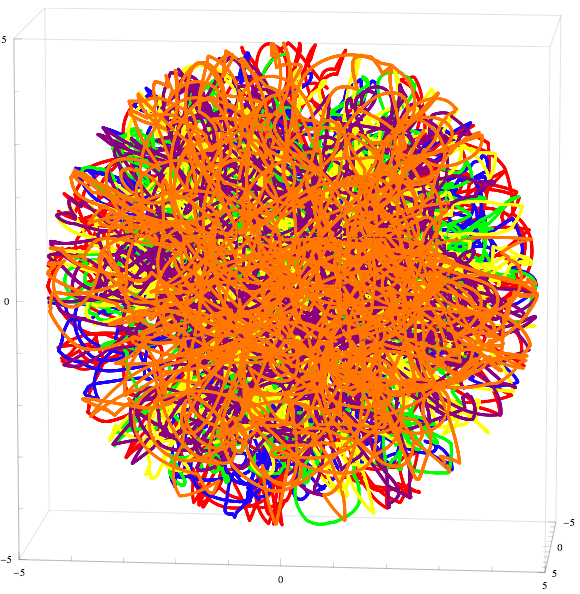} &
\includegraphics[height=0.25\hsize]{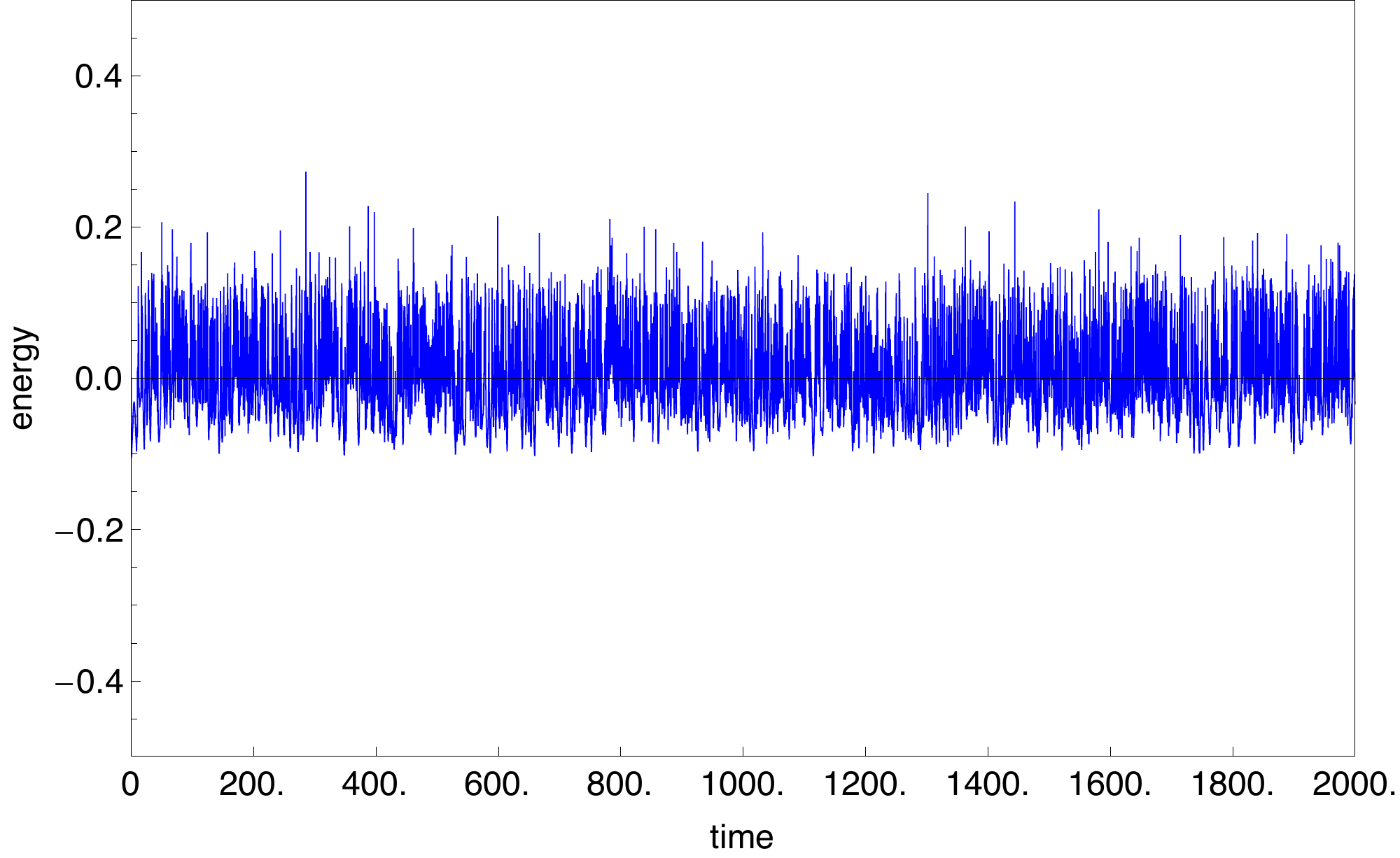}
\\ (a) & (b) & (c) 
\end{tabular}
\label{fig:LJ}
\caption{{\bf Lennard-Jones Example:} In these figures we examine the behavior of  St\"{o}rmer-Verlet-based GVI on a system of six particles in $\mathbb{R}^3$ connected sequentially by rods and interacting in a Lennard-Jones potential.
In {\bf (a)} we take a snapshot of the system configuration, {\bf (b)} traces out the entire trajectory of all particle centers over $2000$ units of simulation time, stepped at $h=10^{-2}$,  while {\bf (c)} plots the change in energy over the same trajectory.}
\end{figure}

\subsubsection{ Interacting Constrained Particles under Lennard-Jones}

Finally, motivated by potential applications in Molecular Dynamics~\citep{Frenkel01}, we consider the example of six spheres in $\mathbb{R}^3$, i.e., $\vc q = (\vc q_1^T, ..., \vc q_6^T)^T \in \mathbb{R}^{18}$,  connected sequentially by bilateral rod constraints, 
\begin{align}
\vc f_{i}(\vc q) = \parallel \vc q_i - \vc q_{i+1} \parallel - \> \> l = 0.
\end{align}
Sphere-pairs are constrained by non-overlap conditions,
\begin{align}
\vc g_{i,j}(\vc q) = \parallel \vc q_i - \vc q_j \parallel - 2 r \geq 0,
\end{align} 
interact in a Lennard-Jones potential,
\begin{align}
V(\vc q) = \sum_{i \neq j} 4 \epsilon \Big[ \Big(\frac{\sigma}{\parallel \vc q_i -\vc q_j \parallel} \Big)^{12} -  \Big( \frac{\sigma}{\parallel \vc q_i -\vc q_j \parallel} \Big)^6 \Big],
\end{align}
and are further constrained to lie in the interior of a bounding sphere of radius $r_B$ so that additionally,
\begin{align}
\vc g_{i,b}(\vc q) = -\parallel \vc q_i \parallel -  r  + r_b \geq 0.
\end{align} 
In our simulation we employ reduced time units, $(m \sigma^2/\epsilon)^{1/2}$,  with the mass of all particles set to unity, $\sigma/2 = r = 0.3$, $r_b= 5$, and $l = 1$. Figure \ref{fig:LJ} (a) shows a snapshot of the system configuration. Figure \ref{fig:LJ} (b) traces out the entire trajectory computed by St\"{o}rmer-Verlet-based GVI of all particle centers over $2000$ units of simulation time, stepped at $h=10^{-2}$,while Figure \ref{fig:LJ} (c) plots the change in energy over the same trajectory.

\section{Discussion and Conclusion}

We have introduced a fully nonsmooth, discrete Hamilton's Principle and an associated DELI integration formulation that enables the generation of VI-based, geometric methods for the numerical integration of generalized inequality/equality-constrained, nonsmooth Hamiltonian systems. Adding additional structure to this framework we have obtained the first family of DELI-based integration methods: GVIs. By construction the GVI methods proposed here preserve momentum and equality constraints. They enforce simultaneous inequality constraints, obtain smooth unilateral motion along constraint boundaries and allow for both nonsmooth and smooth boundary approach and exit trajectories. Finally, we have shown that the proposed GVIs are symplectic over smooth-trajectory intervals and are observed to conserve energy approximately throughout.

We have validated these properties in our numerical experiments and compared them against 
corresponding
direct-substitution and collision-integration schemes in the literature. We have shown that our proposed methods resolve systems subject to multiple persistent, frequent and/or simultaneously active constraints. We have further tested these methods on difficult scenarios, where both smooth and nonsmooth active constraint modes occur simultaneously with high frequency; generally, in many of the examples presented above, during the majority of time steps. 

In many of the above examples, potentially difficult, nonlinear, numerical optimization problems were solved.  In particular, we note that the complementarity conditions in our Discrete Smooth Integrator are, unlike more standard KKT-type complementarity conditions, asymmetric with respect to constraint and constraint gradient evaluations. As such, equations (\ref{eq:ds1}) -- (\ref{eq:ds4}), do not correspond to the optimality conditions of standard nonlinear optimization problems~\citep{Bertsekas} and instead require customized numerical solution strategies. In this work we have focused on investigating the challenges of inequality-constrained, numerical integration and then on developing and validating our proposed methods.  In future work we will explore efficient computational techniques for the scalable solution of these DELI-derived optimization problems. In practice, however, for the numerical experiments discussed here, we have found that relatively simple customized approaches were possible that allowed for efficient solutions, while preliminary experiments investigating the possibility of extending these approaches to larger-scale problems are promising.

Further, we look forward to examining the potential trade-offs between energy behavior and various degrees of constraint enforcement guarantees. We also expect that a rigorous constraint error analysis, based on choice of predictor, should be a useful complement to the GVI methods proposed here. Finally, we note that GVI methods are the first of many possible instantiations of DELI methods. The geometric integration of inequality-constrained systems clearly remains a challenging area of investigation and we hope that the DELI framework will lead to further fruitful explorations, and the development of new complementary numerical methods.

\bibliographystyle{plainnat}
\bibliography{/Users/danny/Desktop/bib/bib}

\appendix

\section{The Discrete Generalized Reflection Operator}
\label{sec:discrete_refl}

For an arbitrary choice of energy function, $E(\vc q, \vc p)$, we invoke the variational model from \citet{KaufmanPaiGrinspun10}, in which a generalized reflection operator is interpreted as a sequence of energy preserving, constraint satisfying projections.  Following their time-continuous case derivation we then generate the corresponding discrete generalized reflection analogue as the discrete, energy preserving, constraint satisfying projection sequence defined in Algorithm 2 below.

\begin{algorithm}
\caption{${\tt Discrete\_Generalized \_ Reflection}(\vc q, \vc p, \mathbb{K}, E)$} \label{alg:jump}
\begin{algorithmic}[1]
\While{{true}}
\State $\mathbb V \leftarrow \emptyset$
\For{$ k \> \> \> \text{in} \> \> \> \mathbb{K}$}
\If{$\nabla g_{_k}(\vc q)^T \vc M^{-1} \vc p < 0$ } 
\State $\mathbb V \leftarrow k$
\EndIf
\EndFor
\If{$\mathbb V \neq \emptyset$}
\State $\lambda \leftarrow  \Big\{ \vc y : \> \> \>  E(\vc q, \vc p + \vc G_{\mathbb V}(\vc q) \vc y) =  E(\vc q, \vc p), \> \> \> \vc G_{\mathbb V}(\vc q)^T \vc M^{-1} \big( \vc p + \vc G_{\mathbb V}(\vc q) \vc y \big) \geq 0, \> \> \> \vc y \geq 0  \Big\}$
\State $\vc p \leftarrow \vc p + \vc G_{\mathbb V}(\vc q) \lambda$
\Else 
\State \Return $\vc p$
\EndIf
\EndWhile
\end{algorithmic}
\end{algorithm}

\begin{theorem}
At termination, Algorithm 2 satisfies all discrete jump conditions. 
\end{theorem}

\begin{proof}
By construction, each iteration satisfies conditions (\ref{eq:discr_jump2}) and (\ref{eq:discr_jump3}). Termination guarantees condition (\ref{eq:discr_jump1}) .
\end{proof}

When the energy function is given by the separable continuous Hamiltonian, e.g., $E = H$, the discrete generalized reflection operator is identical to the time-continuous case considered in \citet{KaufmanPaiGrinspun10}. In this case each solve of line 9, in the above pseudo-code, is obtained by the non-negative, least-squares solution of the normal equation:
\begin{align}
 \vc G_{\mathbb V}(\vc q)^T \vc M^{-1}  \vc G_{\mathbb V}(\vc q) \lambda = -2  \vc G_{\mathbb V}(\vc q)^T \vc M^{-1} \vc p, \quad \lambda \geq 0.
\end{align}
More generally, whenever the Energy function is quadratic in $\vc p$, e.g., as in the St\"{o}rmer-Verlet's second-order numerical Hamiltonian~\citep{Bond07}, the solution to line 9 is likewise given by a non-negative, linear least-squares system. Finally, for possible energy functions containing higher order polynomials in $\vc p$, more complex, variational treatments are required.  

Finally, we note that, as discussed above, GVI is agnostic to the choice of reflection model \emph{as long as} a \emph{well-posed} solution exists that satisfies the discrete jump conditions, (\ref{eq:discr_jump1}) - (\ref{eq:discr_jump3}). As in the case of the Generalized Reflection operator, alternate existing multi-impact models can be similarly discretized to do so. As a particularly compact example, Moreau's elastic multi-impact solution~\citep{Moreau88} can be discretized to obtain the energy preserving discrete reflection given in Algorithm 3 below.

\begin{algorithm}
\caption{${\tt Discrete\_Moreau \_ Reflection}(\vc q, \vc p, \mathbb{K}, E)$} \label{alg:m_jump}
\begin{algorithmic}[1]
\State $\lambda \leftarrow  \Big\{ \vc y : \> \> \>  E(\vc q, \vc p + \vc G_{\mathbb K}(\vc q) \vc y) =  E(\vc q, \vc p), \> \> \> \vc G_{\mathbb K}(\vc q)^T \vc M^{-1} \big( \vc p + \vc G_{\mathbb K}(\vc q) \vc y \big) \geq 0, \> \> \> \vc y \geq 0  \Big\}$
\State $\vc p \leftarrow \vc p + \vc G_{\mathbb K}(\vc q) \lambda$
\State \Return $\vc p$
\end{algorithmic}
\end{algorithm}

\section{Direct Substitution  Schemes} 
\label{sec:direct_methods_proof}

As discussed above, an extended value potential for inequality constrained systems can be composed by concatenating the unconstrained Hamiltonian system's potential with the extended value indicator function.
Following this observation, the \emph{direct-substitution} approach in symplectic methods~\citep{Kane99,Stewart00,Pandolfi02,Deuflhard07}
is to compose direct extensions by treating the nonsmooth constraint force, $-\partial I_{\vc A} (\vc q)$, as simply an additional force to be handled by each numerical method in the usual manner.

\subsection{Direct Integration Examples}
\label{sec:directEx}
To illustrate this approach, in the following, we will present examples using two popular integration approaches: the implicit midpoint method and the Newmark family of integration algorithms.

\subsubsection{Implicit Midpoint Method}   

The implicit midpoint method is given by the discrete, momentum, phase-space pair
\begin{align}
\vc q^{t+1} &= \vc q^{t} + \frac{h}{2} \vc M^{-1} ( \vc p^{t+1} + \vc p^t), \\
\vc p^{t+1} &= \vc p^t - h \nabla V(\frac{ \vc q^{t+1} + \vc q^{t}}{2}).
\end{align}

A direct-substitution, nonsmooth-constrained extension to the midpoint rule is generated by adding the extended value indicator to the unconstrained system's potential. Applying a generalized gradient, this gives  

\begin{align}
\label{eq:standard_nonsmooth_midpoint1}
\vc q^{t+1} &= \vc q^{t} + \frac{h}{2}  \vc M^{-1} ( \vc p^{t+1} + \vc p^t), \\
\label{eq:standard_nonsmooth_midpoint2}
\vc p^{t+1} &\in \vc p^t - h \nabla V(\frac{ \vc q^{t+1} + \vc q^{t}}{2}) - h \partial I_{\vc A} (\frac{ \vc q^{t+1} + \vc q^{t}}{2}) .
\end{align}

An alternate, end-point constrained variant of implicit midpoint can be obtained by simply evaluating constraint forces at the end of the time-step~\citep{Stewart00}. This leads to the modified midpoint method momentum update 
\begin{align}
\label{eq:standard_nonsmooth_midpoint3}
\vc p^{t+1} &\in \vc p^t - h \nabla V(\frac{ \vc q^{t+1} + \vc q^{t}}{2}) - h \partial I_{\vc A} (\vc q^{t+1} ).
\end{align}

\subsubsection{Newmark Methods}

The Newmark family of integration schemes are given by the discrete, velocity, phase-space pair

\begin{align}
\label{eq:newmark_1} \vc q^{t+1} &= \vc q^{t} + h \dot{\vc q}^{t} -
\frac{h^2}{2} \Bigl[ \bigl(1 - 2\beta \bigr) \> \> \vc M^{-1}\nabla V({\vc q}^{t})
+2 \beta \> \> \vc M^{-1} \nabla V({\vc q}^{t+1}) \Bigl],\\
\label{eq:newmark_2} \dot{\vc q}^{t+1} &= \dot{\vc q}^{t} - h \Bigl[
\bigl(1 - \gamma \bigr)\> \> \vc M^{-1} \nabla V({\vc q}^{t}) + \gamma \> \>
\vc M^{-1} \nabla V({\vc q}^{t+1}) \Bigl].
\end{align}
Recall that Newmark methods vary $\beta \in [0, \frac{1}{2}]$ and $\gamma \in [0,1]$.
Newmark is second order accurate if and only  if $\gamma = \frac{1}{2}$, otherwise it is only consistent. Implicit Newmark with $\beta = \frac{1}{4}$ and $\gamma = \frac{1}{2}$ gives the implicit trapazoidal method,
while $\beta = 0$ and $\gamma = \frac{1}{2}$ gives explicit Newmark.

A direct-substitution, nonsmooth-constrained extension of the Newmark Family is generated by adding the extended value indicator to the unconstrained system's potential. Applying a generalized gradient, this gives  
\begin{align}
\label{eq:newmark_1_nonsmooth} \vc q^{t+1} &\in \vc q^{t} + h \dot{\vc q}^{t} -
\frac{h^2}{2} \Bigl[ \bigl(1 - 2\beta \bigr) \> \> \vc M^{-1} \Big( \nabla V({\vc q}^{t}) + \partial I_{\vc A}({\vc q}^{t}) \Big)
+2 \beta \> \> \vc M^{-1} \Big( \nabla V({\vc q}^{t+1}) + \partial I_{\vc A}({\vc q}^{t+1})  \Big) \Bigl],\\
\label{eq:newmark_2_nonsmooth} \dot{\vc q}^{t+1} &\in \dot{\vc q}^{t} 
- h \Bigl[\bigl(1 - \gamma \bigr)\> \> \vc M^{-1} \Big( \nabla V({\vc q}^{t}) + \partial I_{\vc A}({\vc q}^{t}) \Big) + \gamma \> \>
\vc M^{-1} \Big( \nabla V({\vc q}^{t+1}) + \partial I_{\vc A}({\vc q}^{t+1}) \Big) \Bigl].
\end{align}

An Implicit/Explicit (IMEX) variant of this extension was developed by \citet{Kane99} in which both $\beta$ and $\gamma$ are enforced as fully implicit values (i.e., $\beta = \frac{1}{2}, \gamma = 1$) for \emph{just} the nonsmooth portion of the composite potential. This generates a family of  nonsmooth IMEX Newmark integrators given by 

\begin{align}
\label{eq:newmark_1_kane} \vc q^{t+1} &\in \vc q^{t} + h \dot{\vc q}^{t} -
\frac{h^2}{2} \Bigl[ \bigl(1 - 2\beta \bigr) \> \> \vc M^{-1}\nabla V({\vc q}^{t})
+2 \beta \> \> \vc M^{-1} \nabla V({\vc q}^{t+1}) \Bigl]- \frac{h^2}{2} \vc M^{-1}  \partial I_{\vc A}({\vc q}^{t+1}),\\
\label{eq:newmark_2_kane} \dot{\vc q}^{t+1} &\in \dot{\vc q}^{t} 
- h \Bigl[\bigl(1 - \gamma \bigr)\> \> \vc M^{-1} \nabla V({\vc q}^{t})+ \gamma \> \>
\vc M^{-1} \nabla V({\vc q}^{t+1}) \Bigl] 
- h  \vc M^{-1}  \partial I_{\vc A}({\vc q}^{t+1}).
\end{align}

\subsection{Variational Interpretation of Direct Nonsmooth Integrators: Optimization Forms}

Given the explicitly variational definition of the generalized gradient~\citep{RockWets98}, direct, nonsmooth constrained extensions of symplectic methods, as described above, generally reduce to constrained, nonlinear minimizations of the form 
\begin{align} 
\label{eq:ref_min}
\vc q^{t+1} = \argmin_{\vc q} \Big\{ \> \> e(\vc q^t, \vc q, h) : \> \vc f(\vc q^t, \vc q, h) \in \vc A \> \Big\}.
\end{align}

Interpreted in this equivalent form, symplectic methods, applied to Hamiltonian systems, generate an objective function, $e$, generally polynomial in $\vc q$, and a vector-valued, constrained configuration function, $\vc f$, linear in $\vc q$.  

\subsubsection{Implicit Midpoint Example}  

As an example, consider the nonsmooth midpoint rule. We substitute Equation (\ref{eq:standard_nonsmooth_midpoint2}) into  (\ref{eq:standard_nonsmooth_midpoint1}) to obtain the DI 
\begin{align}
\label{eq:optimization_nonsmooth_midpoint}
\vc q^{t+1}  \in  \vc q^t + h \vc M^{-1} \vc p^t - \frac{h^2}{2} \vc M^{-1} \nabla V(\frac{\vc q^t + \vc q^{t+1}}{2}) - \frac{h^2}{2}  \vc M^{-1} \partial I_{\vc A} (\frac{\vc q^t + \vc q^{t+1}}{2}). 
\end{align}
Applying the definition of the generalized gradient, we then obtain an equivalent local, nonlinear minimization that corresponds to the template given by (\ref{eq:ref_min}),
\begin{align}
\vc q^{t+1} = \argmin_{\vc q} \Bigl\{ \frac{1}{2}   \vc q^T \vc M \> \vc q - \vc q^T \big( \vc q^t + h \vc M^{-1} \vc p^t \big) + h^2 V(\frac{\vc q^t + \vc q}{2}) : (\frac{\vc q^t + \vc q}{2}) \in \vc A \Bigr\}.
\end{align}
Rearranging (\ref{eq:standard_nonsmooth_midpoint1}) the corresponding midpoint method, momentum update is then given by 
\begin{align}
\vc p^{t+1} =  \frac{2}{h} \vc M ( \vc q^{t+1} - \vc q^{t} )  - \vc p^t.
\end{align}
Newmark and other direct methods follow similarly.

\subsection{Direct Method Behavior}

Unfortunately, many of the characteristic advantages of symplectic methods are lost in the above direct nonsmooth treatment. The chief observation is, in the smooth setting, the good 
behavior of symplectic methods is ascribed, via backwards error analysis, to the existence of trajectory shadowing \emph{numerical} Hamiltonians~\citep{Hairer03}. More specifically, the flow of fixed-step, symplectic methods follow, up to truncation, a numerical Hamiltonian system that is constructed via a series expansion. Because the numerical Hamiltonian is obtained under a smoothness assumption, which does not hold in the nonsmooth case, standard backwards error analysis guarantees no longer hold for the inequality constrained systems~\citep{Stewart00,Bond07}.

Despite these issues, during free motion (i.e., when constraints are inactive), these nonsmooth constrained symplectic methods trivially reduce to their corresponding unconstrained method and are thus still guaranteed to shadow numerical Hamiltonians. At each constraint boundary event, however, this guarantee is lost, and the solution obtained no longer matches the shadow trajectory. 

In practice, it is exactly at these boundary events that such direct methods obtain a wide variety of incorrect and undesirable behaviors.  Constraint enforcement generates dissipative trajectories for some direct methods, energy growth in others, and in many cases produces effectively nondeterministic restitution behavior both when varying time-step sizes across simulations and within a single fixed step-size simulation~\citep{Stewart00}. 

\subsection{No Energy Conservation}
\label{sec:no_energy_conservation}
Considering (\ref{eq:ref_min}) we, in particular, note that, perhaps surprisingly, a large reason for these difficulties stems directly from the underlying variational structure of these methods. 

\begin{theorem}
All direct, one-step, nonsmooth symplectic methods that can be posed in the form of (\ref{eq:ref_min}), can not, in general, preserve, either approximately, nor exactly, the energy of the underlying unconstrained symplectic method. 
\end{theorem}

\begin{proof}
For all such methods,  (\ref{eq:ref_min}) can be rewritten, with respect to the constraint functions, as
\begin{align}
\vc q^{t+1} = \argmin_{\vc q} \Big\{ \> \> e(\vc q^t, \vc q, h) : \> \vc g\Big( (1-\alpha)\vc q^t + \alpha \vc q \Big) \geq 0 \> \Big\},
\end{align}
for some fixed $\alpha \in (0,1]$.
 
We then note that we retrieve the solution of the underlying, unconstrained symplectic method by solving for $\vc q$ in 
\begin{align}
\nabla_{\scriptsize \vc q} e(\vc q^t, \vc q, h) = 0.
\end{align}
The constrained variant will return this solution for any step in which the unconstrained solution returns a non-violating configuration.
More generally, formulating the Karush-Kuhn-Tucker (KKT) first order optimality conditions~\citep{Bertsekas}, for the above minimization, we obtain, for optimal $\vc q^{t+1} = \vc q$,
\begin{align}
\nabla_{\scriptsize \vc q} e(\vc q^t, \vc q, h) - \alpha \nabla \vc g \Big( (1-\alpha)\vc q^t + \alpha \vc q \Big) \lambda &= 0,\\
0 \leq \lambda \perp \vc g \Big( (1-\alpha)\vc q^t + \alpha \vc q \Big) &\geq 0.
\end{align}
We then consider any time step in which an unconstrained solution violates constraints. The above KKT system implies that constraint forces are applied along constraint gradients and, more specifically, are of the form $\alpha \nabla \> \vc g \> \> \lambda, \> \> \> \lambda \geq 0$. 

The complementarity term of the above KKT system then additionally requires, component-wise, that a constraint force magnitude, 
applied along the gradient $\nabla g_i$, can be nonzero (and thus enforce the corresponding constraint, $g_i $) if and only if 
\begin{align}
g_i \Big( (1-\alpha)\vc q^t + \alpha \vc q^{t+1} \Big)  = 0. 
\end{align}
Thus constraining forces can only be applied along constraint gradients (to enforce feasibility) if the
position at which the constraint force is evaluated, $(1-\alpha)\vc q^t + \alpha \vc q^{t+1}$, 
lies \emph{exactly} on the corresponding constraint's boundary. Since these methods demand constraint enforcement by construction, they implicitly enforce this additional condition as a secondary constraint. In turn, this implicit constraint overrides the underlying symplectic structure of the base, unconstrained method. In particular, the magnitude of the corresponding constraint force, that achieves this secondary constraint, scales independently of time-step size and energy. Instead, these constraint force magnitudes scale with constraint geometry and state, so that an arbitrarily large perturbation may be applied to the discrete system's energy and momentum in order to place the configuration at which constraint forces are evaluated, $(1-\alpha)\vc q^t + \alpha \vc q^{t+1}$, on the admissible set boundary.
\end{proof}

\section{The Discrete-Smooth Integrator}
\label{sec:last}
We now show symplecticity and momentum conservation for the Discrete-Smooth Integrator (DSI) by proving Theorem 5.1.
\begin{proof}
As a preliminary, starting at time $t^+$, we first define an \emph{Oracle Set} composed of constraint indices  that will be active at the end of the current time step, $\mathbb{O}( \vc q^{t+1}) \defeq \{ i : g_i(\vc q^{t+1}) = 0\} $.  
The intersection of the oracle set with the \emph{Smooth Set},  $\mathbb{S}(\vc q^t,\vc p^{t^+})$, defined in (6.1), gives $\mathbb{I} \defeq \mathbb{O}( \vc q^{t+1}) \> \> \> \cap \> \> \> \mathbb{S}(\vc q^t,\vc p^{t^+})$.

Now consider the optimal $\lambda^-$ and $\mu$ that obtain $(\vc q^{t+1}, \vc p^{t+1})$ in (6.3) -- (6.6). We then observe that setting $\gamma = \lambda^-$ and $\delta = \mu$ we also satisfy the following equality constrained integration system
\begin{align}
\label{eq:ds1B}
 D_1L_d (\vc q^t, \vc q^{t+1}) + \vc p^t + \vc G_{\mathbb{I}}(\vc q^t) \gamma &= 0,\\
 \label{eq:ds2B}
\vc g_{\mathbb{I}}(\vc q^{t+1}) &= 0,\\
 \label{eq:ds3B}
\vc p^{t+1} &= D_2 L_d (\vc q^{t}, \vc q^{t+1}) + \vc G_{\mathbb{I}}(\vc q^{t+1}) \> \> \delta,\\
 \label{eq:ds4B}
\vc G_{\mathbb{I}}(\vc q^{t+1})^T \vc M^{-1} \vc p^{t+1} &= 0.
\end{align} 
In particular, plugging $\gamma = \lambda^-$ and $\delta = \mu$  into the above system gives \emph{exactly} the same $(\vc q^{t+1}, \vc p^{t+1}$) solution as the DSI system. 

Next we observe that Equations (\ref{eq:ds1}) through (\ref{eq:ds4}) give the position-momentum form of the equality constrained discrete Hamiltonian map of~\citet[(3.5.2a) -- (3.5.2d)] {MarsdenWest01A}. Symplecticity and momentum conservation then follow directly by noting that restricting DSI steps to constraints indexed in $\mathbb{S}(\vc q^t,\vc p^{t^+})$ guarantees initial position and momentum feasibility~\citep[(3.5.1)] {MarsdenWest01A} for the corresponding equality constrained discrete Hamiltonian maps, at all steps. 
\end{proof}

\end{document}